\documentclass[11pt]{article}

\usepackage[a4paper,margin=2.45cm]{geometry}
\usepackage{amsmath,amssymb,amsthm,amsfonts,amstext,bm,mathtools,amsopn}
\usepackage{booktabs,enumitem,microtype,xcolor}
\usepackage{multicol}
\usepackage{algorithm}
\usepackage{algorithmic}
\usepackage{boxedminipage}
\usepackage{fancybox}
\usepackage{cases,supertabular}
\usepackage{multirow,makecell}
\usepackage{graphicx}
\usepackage{array}
\usepackage{listings}
\usepackage{subfigure}
\usepackage{enumerate}
\usepackage[sort&compress]{natbib}
\usepackage[colorlinks=true,allcolors=blue!55!black]{hyperref}
\usepackage{etoolbox}

\bibpunct[, ]{[}{]}{,}{n}{}{,}

\DeclareGraphicsExtensions{.eps,.pdf,.png,.jpg}
\DeclareUnicodeCharacter{2000}{ }
\DeclareUnicodeCharacter{2011}{\mbox{-}}
\DeclareUnicodeCharacter{2013}{--}

\theoremstyle{plain}
\newtheorem{theorem}{Theorem}
\newtheorem{lemma}{Lemma}
\newtheorem{assumption}{Assumption}
\theoremstyle{remark}
\newtheorem{remark}{Remark}

\renewenvironment{proof}[1]{%
	\par\noindent\textit{#1.}\enspace\ignorespaces
}{%
	\par\medskip
}
\newcommand{\Halmos}{\hfill\ensuremath{\square}}

\newcommand{\compactalignspacing}{%
	\setlength{\abovedisplayskip}{0.15em plus 0.05em minus 0.03em}%
	\setlength{\belowdisplayskip}{0.15em plus 0.05em minus 0.03em}%
	\setlength{\abovedisplayshortskip}{0pt plus 0.03em}%
	\setlength{\belowdisplayshortskip}{0.12em plus 0.03em minus 0.03em}}
\AtBeginEnvironment{align}{\compactalignspacing}
\AtBeginEnvironment{align*}{\compactalignspacing}

\title{Restart-Free Oracle-Efficient Methods for Strongly Convex Composite Optimization}
\author{%
	Xinming Wu\\[-0.15em]
	{\small Shanghai Key Laboratory for Contemporary Applied Mathematics, School of Mathematical Sciences,}\\[-0.15em]
	{\small Fudan University, Shanghai 200433, China, \texttt{wuxinming@fudan.edu.cn}}\\[0.8em]
	Zi Xu\thanks{Corresponding author.}\\[-0.15em]
	{\small Department of Mathematics, College of Sciences, Shanghai University, Shanghai, 200444, China,}\\[-0.15em]
	{\small \texttt{xuzi@shu.edu.cn}.}\\[0.8em]
	Huiling Zhang\\[-0.15em]
	{\small LSEC, ICMSEC, Academy of Mathematics and Systems Science,}\\[-0.15em]
	{\small Chinese Academy of Sciences, Beijing 100190, China, \texttt{zhanghl@amss.ac.cn}}
}
\date{}

\begin{document}

	\maketitle

	\begin{abstract}
		We consider strongly convex composite optimization problems whose objective is the sum of a general smooth component, a general nonsmooth component, and a relatively simple nonsmooth term. Although existing first-order methods can achieve optimal convergence guarantees under strong convexity, they typically rely on explicit restart schemes, leading to multiphase algorithms with increased structural and implementation complexity. To address this issue, we develop a restart-free stochastic gradient sliding method that incorporates strong convexity directly through a novel parameter selection strategy. We establish that the proposed method computes an $\epsilon$-solution with $\mathcal{O}\bigl(\log(1/\epsilon)\bigr)$ gradient evaluations of the smooth component and $\mathcal{O}(1/\epsilon)$ stochastic subgradient evaluations of the general nonsmooth component. These bounds match the corresponding complexity guarantees of existing multiphase, restart-based methods, while eliminating the need for explicit restarts. Furthermore, when subgradients of the general nonsmooth component are difficult to compute, but a compatible smooth approximation and its first-order oracle are available, we develop a restart-free accelerated stochastic smoothed gradient sliding (RF-ASSGS) method. RF-ASSGS requires $\mathcal{O}\bigl(\log(1/\epsilon)\bigr)$ gradient evaluations of the smooth component and $\mathcal{O}\bigl(1/\sqrt{\epsilon}+\sigma^2/\epsilon\bigr)$ stochastic-oracle evaluations. Overall, our results demonstrate that restart-free gradient sliding methods can preserve the optimal convergence guarantees of their restart-based counterparts for strongly convex objectives, while substantially simplifying the algorithmic structure and implementation.
	\end{abstract}

	\noindent\textbf{Funding.} This work is supported by National Key  R \& D Program of China (Nos. 2025YFA1017801 and 2025YFA1017800), National Natural Science Foundation of China under the grant 12471294 and the Postdoctoral Fellowship Program of CPSF under Grant Number GZB20240802 and 2024M763470.

	\medskip

	\section{Introduction}
	In this paper, we consider the following class of composite convex optimization problems:
	\begin{align}\label{prob}
		\min_{x \in X} \; \Psi(x) := f(x) + h(x) + \chi(x),
	\end{align}
	where $X \subseteq \mathbb{R}^n$ is a closed convex set, $f$ is a smooth convex function, $h$ is a nonsmooth convex function, and $\chi$ is a relatively simple convex function whose proximal mapping can be computed efficiently. Problems of the form \eqref{prob} arise in a wide range of data analysis applications, including total variation regularization \cite{Rudin, Lan2022}, sparse logistic regression \cite{Berkson}, low-rank tensor recovery \cite{Kolda, Tomioka}, graph regularization \cite{Jacob, Tibshirani2005}, and the Lasso problem \cite{Jacob, Mairal, Tibshirani}.
	
	In addition to stochastic subgradients of $h$, we also study the setting in which subgradients of $h$ are difficult to compute, but $h$ admits a computable family of smooth convex approximations. For example, for a max-form function, Nesterov's smoothing technique adds a strongly convex regularizer to the inner maximization problem, thereby ensuring the uniqueness of its maximizer and differentiability of the resulting approximation. When the regularized inner problem can be solved efficiently, the gradient of the smoothed function can be obtained directly from this unique maximizer \cite{Nesterov05}. The smoothing parameter balances the approximation error against the Lipschitz constant of the smoothed gradient. Max-form models arise in a variety of applications, including reinforcement learning \cite{Du}, empirical risk minimization \cite{Kovalev}, network flow optimization \cite{Zargham}, decentralized distributed optimization \cite{Arjevani, Kovalev2020, Ye}, and optimal transport \cite{Peyre}.
	
	When the objective function $\Psi(x)$ in \eqref{prob} is smooth, that is, when $h \equiv 0$ and $\chi \equiv 0$, problem \eqref{prob} reduces to minimizing a smooth convex function over the convex set $X$. For this fundamental problem class, Nesterov's Accelerated Gradient (AG) method \cite{Nesterov1983, Nesterov2018} achieves the optimal convergence rates of $\mathcal{O}(1/\sqrt{\epsilon})$ iterations for convex problems and $\mathcal{O}(\log(1/\epsilon))$ iterations for strongly convex problems. These rates are known to be optimal \cite{Nesterov2018}, and related extensions and complexity analyses have been developed in \cite{Ghadimi2016, Lan2012, Tseng}.
	
	For the nonsmooth convex problem \eqref{prob}, a standard setting assumes that the proximal mapping of $h$ can be computed efficiently. In this case, the proximal gradient method has an iteration complexity of $\mathcal{O}(1/\epsilon)$ \cite{Nesterov2018}, which can be improved to $\mathcal{O}(1/\sqrt{\epsilon})$ through multistep acceleration schemes \cite{Tseng}. However, these methods become impractical when $h$ is a general nonsmooth function whose proximal mapping is difficult or expensive to evaluate.
	
	An alternative is to access $f$ and $h$ separately through their first-order oracles: a gradient for $f$ and a subgradient, denoted by $h'$, for $h$. Under this oracle model, Nemirovskij et al. \cite{Nemirovskij2009} showed that a properly modified stochastic approximation method requires $O((L^2 + M^2 + \sigma^2)/\epsilon^2)$ iterations to compute an $\epsilon$-solution $\bar{x} \in X$ satisfying $\Psi(\bar{x})-\Psi^* \le \epsilon$. Here, $\sigma$ denotes the stochastic noise, $L$ is the Lipschitz constant of $\nabla f$, and $M$ is a constant satisfying
	\[
	h(x) \le h(y) + \langle h'(y), x - y \rangle + M\|x - y\|, \quad \forall x, y \in X.
	\]
	Subsequently, Juditsky et al. \cite{Juditsky} proposed a stochastic mirror-prox method that bounds the number of oracle calls to $\nabla f$ and $h'$ by $\mathcal{O}\big(\frac{L}{\epsilon} + \frac{M^2+\sigma^2}{\epsilon^2}\big)$. Lan \cite{Lan2012} further advanced this line of research by developing an accelerated stochastic approximation method, which improves the bound to $\mathcal{O}\big(\sqrt{\frac{L}{\epsilon}} + \frac{M^2+\sigma^2}{\epsilon^2}\big)$. As observed in \cite{Lan2012}, this complexity is unimprovable when only first-order information about the sum $f+h$ is available through a single black-box oracle.
	
	For the max-form special case of problem \eqref{prob}, Nesterov \cite{Nesterov05} proposed smoothing the nonsmooth component before applying an optimal gradient method. He showed that an $\epsilon$-solution of the associated bilinear saddle-point problem can be obtained in at most $O\!\left(\sqrt{\frac{L}{\epsilon}} + \frac{\|K\|}{\epsilon}\right)$ iterations. This complexity was subsequently shown to be theoretically unimprovable \cite{Ouyang}. Motivated by this result, substantial effort has been devoted to developing first-order methods that exploit the saddle-point structure of this special problem class. Representative examples include mirror-prox methods \cite{Chen, Juditsky, Nemirovskij2004}, primal-dual methods \cite{Chen2014, Esser, He}, and their equivalent formulations as the alternating direction method of multipliers \cite{HeB, Monteiro, Ouyang2015}.
	
	In many practical applications, evaluating the gradient $\nabla f(x)$ is substantially more expensive than computing a subgradient $h'(x) \in \partial h(x)$ or applying the linear operators $K$ and $K^\top$. For example, in total variation regularization and overlapped group lasso, the operator $K$ is typically sparse, whereas $f$ often contains a computationally intensive data-fitting term \cite{Lan_sliding, Lan2022}. This computational asymmetry motivates the development of algorithms that reduce the number of expensive evaluations of $\nabla f$ while preserving the complexity of the less expensive subgradient evaluations of $h'$. To this end, Lan \cite{Lan_sliding} proposed the stochastic gradient sliding (SGS) algorithm and showed that the number of gradient evaluations of $\nabla f$ needed to compute an $\epsilon$-solution of \eqref{prob} can be reduced to $\mathcal{O}\!\left(\sqrt{\frac{L}{\epsilon}}\right)$. Meanwhile, the number of stochastic subgradient evaluations of $h'$ remains $\mathcal{O}\!\left(\sqrt{\frac{L}{\epsilon}} + \frac{M^2+\sigma^2}{\epsilon^2}\right)$. When the smooth component $f$ is $\mu$-strongly convex, Lan \cite{Lan_sliding} further introduced a multiphase stochastic gradient sliding (M-SGS) algorithm based on a restarting strategy. This method achieves the optimal complexity bounds $\mathcal{O}\!\left(\sqrt{\frac{L}{\mu}}\log\bigl(\frac{1}{\epsilon}\bigr)\right)$ for gradient evaluations of $\nabla f$ and $\mathcal{O}\!\left(\frac{M^2+\sigma^2}{\mu\epsilon}\right)$ for subgradient evaluations of $h'$.
	
	For the max-form saddle-point problem, Lan et al. \cite{Lan2022} proposed an accelerated gradient sliding (AGS) method for the case $\chi(x)\equiv0$. They proved that an $\epsilon$-solution can be obtained using at most $O\!\left(\sqrt{\frac{L}{\epsilon}}\right)$ gradient evaluations of $\nabla f$ and $O\!\left(\frac{\|K\|}{\epsilon}\right)$ evaluations of $K$ and $K^\top$. Moreover, when $f$ is $\mu$-strongly convex, they developed a multistage AGS (M-AGS) algorithm and showed that these complexities can be improved to $O\!\left(\sqrt{\frac{L}{\mu}}\log\bigl(\frac{1}{\epsilon}\bigr)\right)$ and $O\!\left(\frac{\|K\|}{\sqrt{\epsilon}}\right)$, respectively.
	
	Although these multiphase and multistage algorithms attain optimal complexity bounds, their reliance on a \emph{restarting strategy} introduces both practical and structural complications. In particular, the restart mechanism requires careful maintenance of algorithmic states, repeated parameter resetting, and the solution of subproblems to prescribed accuracies at each phase. These requirements complicate implementation and make it difficult to incorporate such algorithms into larger optimization frameworks. This motivates the development of simpler \emph{restart-free} methods that attain the same complexity guarantees. Accordingly, we ask the following question:
	\emph{Can one develop restart-free, oracle-efficient algorithms for problem \eqref{prob}, based either on stochastic subgradients of $h$ or on stochastic gradients of a compatible smoothing family, while retaining the strong-convexity complexity bounds established in \cite{Lan_sliding} and \cite{Lan2022}?}
	
	\subsection{Contributions}
	In this paper, we answer the above question affirmatively. Our main contributions are the design and analysis of novel stochastic gradient sliding algorithms that require no explicit restarting mechanism while preserving the oracle complexities of their restart-based counterparts for strongly convex problems. The central innovation is a refined continuous parameter-update strategy that reproduces the benefits of restart phases within a single-loop algorithm.
	
	\begin{itemize}
		\item We develop a \textbf{R}estart-\textbf{F}ree \textbf{S}tochastic \textbf{G}radient \textbf{S}liding (\textbf{RF-SGS}) algorithm for solving problem \eqref{prob} under the assumption that $\chi(x)$ is $\mu$-strongly convex. By using a parameter schedule that evolves continuously throughout the iterations, RF-SGS eliminates the multiphase restarting scheme employed in \cite{Lan_sliding}. The method requires $\mathcal{O}(\sqrt{L/(\mu)}\log(1/\epsilon))$ evaluations of $\nabla f$ and $\mathcal{O}(L/(\mu\epsilon)+\sqrt{L/(\mu)}\log(1/\epsilon))$ stochastic subgradient evaluations of $h$. Thus, the restart mechanism, although useful in the design and analysis of existing methods, is not necessary for attaining the same complexity orders.
		
		\item When direct subgradients of $h$ are difficult to compute but $h$ admits a compatible smoothing family $\{h_\eta\}_{\eta>0}$, we develop a \textbf{R}estart-\textbf{F}ree \textbf{A}ccelerated \textbf{S}tochastic \textbf{S}moothed-\textbf{G}radient \textbf{S}liding (\textbf{RF-ASSGS}) algorithm. Through a continuous parameter-update policy, RF-ASSGS avoids the multistage restarts used by M-AGS in \cite{Lan2022}. We establish that RF-ASSGS achieves $\mathcal{O}(\sqrt{L/(\mu)}\log(1/\epsilon))$ gradient evaluations of $\nabla f$ and $\mathcal{O}(\sqrt{L/(\mu)}\log(1/\epsilon)+\|K\|/\sqrt{\mu\epsilon}+\sigma^2/(\mu\epsilon))$ stochastic-oracle evaluations for $h$ in expectation. If exact first-order information is available, i.e., $\sigma=0$, RF-ASSGS requires $O(\sqrt{L/(\mu)}\log(1/\epsilon)+\|K\|/\sqrt{\mu\epsilon})$ evaluations of $K$ and $K^\top$, matching the corresponding oracle-complexity orders established in \cite{Lan2022}. At first glance, our problem formulation differs from that of M-AGS \cite{Lan2022}: we assume that $f$ is smooth and convex and place strong convexity in the proximal term $\chi$, whereas M-AGS assumes $\chi \equiv 0$ and a strongly convex $f$. This distinction is only representational, since a problem with strongly convex $f$ can be reformulated within our framework without changing the original objective, as detailed in Remark~\ref{remark:2}. Consequently, our method applies whenever strong convexity is carried by either $f$ or $\chi$, and in particular covers the problem class considered by M-AGS.
	\end{itemize}
	
	\subsection{Organization}
	The remainder of this paper is organized as follows. In Section~\ref{sec2}, we present the restart-free stochastic gradient sliding (RF-SGS) algorithm for solving problem~\eqref{prob} and establish its convergence guarantees. Section~\ref{sec:3} introduces the restart-free accelerated stochastic smoothed-gradient sliding (RF-ASSGS) algorithm for nonsmooth functions admitting a compatible smoothing family and analyzes its convergence properties. Numerical experiments demonstrating the performance of the proposed methods are reported in Section~\ref{sec4}. Finally, the paper concludes with a summary and a discussion of directions for future research.
	
	\subsection{Notations}
	We denote by $\|\cdot\|$ an arbitrary norm in $\mathbb{R}^n$ and by $\|\cdot\|_*$ its dual norm. For a convex function $h$, $\partial h(x)$ denotes the subdifferential of $h$ at $x$.
	
	A function $\omega: X \to \mathbb{R}$ is called a \emph{distance generating function} with modulus $\nu>0$ relative to $\|\cdot\|$ if it is continuously differentiable and strongly convex with parameter $\nu$, that is,
	\[
	\langle x - z, \nabla\omega(x) - \nabla\omega(z) \rangle \geq \nu \|x - z\|^2, \quad \forall\, x,z \in X.
	\]
	The associated \emph{Bregman distance} is defined by
	\[
	V(x,z) = \omega(z) - \omega(x) - \langle \nabla\omega(x), z - x \rangle.
	\]
	The strong convexity of $\omega$ implies that $V(x,z)\ge \frac{\nu}{2}\|x-z\|^2$ for all $x,z\in X$. We additionally assume that the chosen Bregman distance grows quadratically:
	$$V(x,z) \leq \frac{1}{2}\|x-z\|^2,~\forall x,z\in X.$$
	
	The indicator function of a set $X$ is denoted by $\delta_X(\cdot)$ and is defined as
	\[
	\delta_X(x) =
	\begin{cases}
		0, & \text{if } x \in X, \\[4pt]
		+\infty, & \text{if } x \notin X.
	\end{cases}
	\]

	\section{Stochastic gradient sliding algorithm without restart}\label{sec2}
	In this section, we propose a restart-free stochastic gradient sliding (RF-SGS) algorithm for solving problem \eqref{prob}. This algorithm extends the SGS framework \cite{Lan_sliding} by eliminating the need for a restarting mechanism. Following the setting in \cite{Lan_sliding}, we focus on the scenario where computing stochastic subgradients of $h$ is significantly cheaper than obtaining exact (deterministic) subgradients. To this end, we assume that first-order information for $h$ is accessed via a stochastic oracle ($\mathcal{SO}$). At each iteration $t$, given an input point $x_{k,t} \in X$, the $\mathcal{SO}$ returns a vector $H(x_{k,t}, \xi_{k,t})$, where $\{\xi_{k,t}\}_{t \geq 1}$ is a sequence of independent and identically distributed (i.i.d.) random variables.
	
	Each iteration of the proposed RF-SGS algorithm consists of two nested loops: an outer loop and an inner loop. In the outer loop, we employ an accelerated proximal gradient step of the form:
	\begin{align}
		\underline{x}_k &= (1-\gamma_k)\bar{x}_{k-1} + \gamma_k x_{k-1}, \label{up-uxk} \\
		\tilde{x}_k &\approx \arg\min_{u \in X} \Big\{
		l_f(\underline{x}_k, u)
		+ h(u)
		+ \chi(u)
		+ \beta_k V(x_{k-1}, u)
		\Big\}, \label{up-xk} \\
		\bar{x}_k &= (1-\gamma_k)\bar{x}_{k-1} + \gamma_k \tilde{x}_k, \label{up-bxk}
	\end{align}
	where $l_f(\underline{x}_k, u) = f(\underline{x}_k) + \langle \nabla f(\underline{x}_k), u - \underline{x}_k \rangle$ is the linearization of $f$ at $\underline{x}_k$, and $V(\cdot,\cdot)$ is a Bregman distance.
	
	In the inner loop, we reuse the same gradient $\nabla f(\underline{x}_k)$ throughout $T_k$ inner updates. The subproblem \eqref{up-xk} is solved approximately via the following stochastic subgradient steps:
	\begin{align*}
		u_{k,t} &= \arg\min_{u \in X} \Big\{
		l_f(\underline{x}_k, u)
		+ \langle H(u_{k,t-1}, \xi_{k,t}), u - u_{k,t-1} \rangle
		+ \chi(u) \\
		&\qquad\qquad\qquad + \beta_k V(x_{k-1}, u)
		+ \beta_k p_{k,t} V(u_{k,t-1}, u)
		\Big\}, \\
		\tilde{u}_{k,t} &= (1-\theta_{k,t}) \tilde{u}_{k,t-1} + \theta_{k,t} u_{k,t},
	\end{align*}
	where $H(u_{k,t-1},\xi_{k,t})$ is a stochastic subgradient of $h$ at the predictable point $u_{k,t-1}$. The complete RF-SGS algorithm is formally described in Algorithm \ref{alg}.
	\begin{algorithm}[t]
		\caption{Restart-free stochastic gradient sliding (RF-SGS) algorithm}
		\label{alg}
		\begin{algorithmic}[1]
			\REQUIRE $x_0\in X$ and an outer-iteration horizon $N\ge1$.  Set $c=\frac{\sqrt{\frac{L}{\mu\nu}}}{1+\sqrt{\frac{L}{\mu\nu}}}$, $\bar{x}_0=x_0$.
			\FOR{$k=1,2,\cdots, N$}
			\STATE Set\par\noindent
			\begin{minipage}{\linewidth}
				\begin{align*}
					\beta_k&=\beta =\frac{L}{\nu} \cdot \frac{1}{1+\sqrt{\frac{L}{\mu\nu}}},\quad \gamma_k=\frac{1}{1+\sqrt{\frac{L}{\mu\nu}}},\quad
					T_k=\left\lceil  \frac{1}{c^{k/2}}\cdot  \frac{(\beta+\mu)(1-c)}{ c(\beta+\mu) -\beta}\right\rceil.
				\end{align*}
			\end{minipage}
			\STATE Compute $\underline{x}_k = (1-\gamma_k) \bar{x}_{k-1} + \gamma_k x_{k-1}$.\\
			\STATE Set $u_{k,0}=\tilde{u}_{k,0} =x_{k-1}$.
			\FOR{$t=1,2,\cdots, T_k$}
			\STATE Set $p_{k,t} = \frac{\beta+\mu}{\beta} \cdot  \frac{1}{c^{k/2}}$, $\theta_{k,t}  = \frac{1- \frac{1}{1+c^{k/2}} }{  1- \frac{1}{(1+c^{k/2})^t} }$.
			\STATE Compute\par\noindent
			\begin{minipage}{\linewidth}
				\begin{align}
					u_{k,t} =&\arg\min_{u\in X} \left\{ l_f (\underline{x}_k, u) +\langle H(u_{k,t-1}, \xi_{k,t}), u-u_{k,t-1}\rangle+\chi (u) \right.\nonumber\\
					&\qquad\qquad\left. + \beta_k V(x_{k-1}, u) +   {\color{black}\beta_k p_{k,t} V(u_{k,t-1}, u)} \right\},\label{update-u}\\
					\tilde{u}_{k,t} =& (1-\theta_{k,t}) \tilde{u}_{k,t-1} + \theta_{k,t}  u_{k,t}.\label{update-tu}
				\end{align}
			\end{minipage}
			\ENDFOR
			\STATE 
			Set $x_k= u_{k,T_k}$, $\tilde{x}_k = \tilde{u}_{k,T_k}$.\\
			\STATE Compute $\bar{x}_k  =  (1-\gamma_k) \bar{x}_{k-1} + \gamma_k\tilde{x}_{k}$.
			\ENDFOR
		\end{algorithmic}
	\end{algorithm}
	
	Note that the RF-SGS algorithm computes the gradient $\nabla f(\underline{x}_k)$ only once per outer loop and reuses it throughout the $T_k$ inner updates. This strategy significantly reduces the total number of gradient evaluations for $f$, which is particularly beneficial when evaluating $\nabla f$ is substantially more expensive than computing stochastic subgradients of $h$. The key distinction between the proposed RF-SGS algorithm and the multi-phase SGS (M-SGS) method in \cite{Lan_sliding} lies in the restart mechanism: M-SGS relies on periodic restarts every fixed number of iterations, while RF-SGS uses a carefully designed, continuous parameter schedule without requiring any restart. Although the analysis below is stated with strong convexity in $\chi$, the equivalent reformulation that transfers strong convexity from $f$ to the proximable component in Reamrk \ref{remark:2} shows that RF-SGS also covers the setting addressed by M-SGS in which $f$ is strongly convex.
	
	\subsection{Complexity analysis}\label{sec3}
	In this section, we establish the convergence properties of the RF-SGS algorithm for solving problem \eqref{prob}.
	
	We first present two technical lemmas that are essential for the convergence analysis. The first lemma characterizes the solution of the proximal projection step \eqref{update-u}.
	\begin{lemma}\label{lemma:1.1}
		If $q$ is a $\mu$-strongly convex function and
		\begin{align}
			u^* = \arg\min_{u\in X} \{ q(u) + \mu_1 V(\tilde{x}, u) + \mu_2 V(\tilde{y}, u)\},
		\end{align}
		then $\forall u \in X$,  we have
		\begin{align}
			q(u^*) + \mu_1 V(\tilde{x}, u^*) + \mu_2 V(\tilde{y}, u^*) &\le q(u) + \mu_1 V(\tilde{x}, u) + \mu_2 V(\tilde{y}, u) \nonumber\\
			& \quad - (\mu_1+\mu_2) V(u^*, u) - \frac{\mu}{2} \|u-u^*\|^2.
		\end{align}
	\end{lemma}
	\begin{proof}{Proof}
		Let $q'(u^*)\in\partial q(u^*)$. The optimality condition gives, for every $u\in X$,
		\begin{align}\label{lemma:1.1-opt}
			\left\langle q'(u^*)+\mu_1[\nabla\omega(u^*)-\nabla\omega(\tilde x)]
			+\mu_2[\nabla\omega(u^*)-\nabla\omega(\tilde y)],u-u^*\right\rangle\ge0.
		\end{align}
		The $\mu$-strong convexity of $q$ and the Bregman three-point identity yield
		\begin{align*}
			q(u)&\ge q(u^*)+\langle q'(u^*),u-u^*\rangle+\frac\mu2\|u-u^*\|^2,\\
			\langle\nabla\omega(u^*)-\nabla\omega(\tilde x),u-u^*\rangle
			&=V(\tilde x,u)-V(\tilde x,u^*)-V(u^*,u),\\
			\langle\nabla\omega(u^*)-\nabla\omega(\tilde y),u-u^*\rangle
			&=V(\tilde y,u)-V(\tilde y,u^*)-V(u^*,u).
		\end{align*}
		Substituting these identities into \eqref{lemma:1.1-opt} and rearranging gives the claimed inequality.
		\Halmos
	\end{proof}
	The second lemma provides a convenient recursive inequality for analyzing convergence of sequences; its proof can be found in Lemma 2 of \cite{Lan_sliding}.
	\begin{lemma}\label{lem2}
		Let $w_k \in (0,1]$ for $k=1,2,\dots$, and let $W_1>0$ be given. Define
		\[
		W_k := (1-w_k) W_{k-1}, \quad k \geq 2 .
		\]
		Assume that $W_k>0$ for all $k \geq 2$, and that the sequence $\{\delta_k\}_{k \geq 0}$ satisfies
		\[
		\delta_k \leq (1-w_k) \delta_{k-1} + B_k, \quad k=1,2,\dots .
		\]
		Then, for any $k \geq 1$, we have
		\[
		\delta_k \leq W_k \left[ \frac{1-w_1}{W_1} \delta_0 + \sum_{i=1}^k \frac{B_i}{W_i} \right].
		\]
	\end{lemma}
	Next, we state several mild assumptions required for the analysis.
	\begin{assumption}\label{ass:f}
		The functions $f: X \to \mathbb{R}$ and $h: X \to \mathbb{R}$ are convex and satisfy
		\begin{align*}
			f(x) &\leq f(y) + \langle \nabla f(y), x-y \rangle + \frac{L}{2} \|x-y\|^2, \quad \forall x, y \in X, \\
			h(x) &\leq h(y) + \langle h'(y), x-y \rangle + M \|x-y\|, \quad \forall x, y \in X,
		\end{align*}
		for some $L\ge\nu\mu>0$ and $M>0$, where $h'(x) \in \partial h(x)$.
	\end{assumption}
	
	\begin{assumption}\label{ass:chi}
		The function $\chi(x)$ is $\mu$-strongly convex.
	\end{assumption}
	
	\begin{assumption}\label{ass:h}
		Let $\mathcal F_{k,t-1}$ denote the $\sigma$-field generated by all oracle samples drawn before $\xi_{k,t}$. For every $\mathcal F_{k,t-1}$-measurable point $u_{k,t-1}\in X$, there exists a constant $\sigma\ge0$ such that
		\begin{align*}
			\mathbb{E}\big[ H(u_{k,t-1}, \xi_{k,t})\mid\mathcal F_{k,t-1} \big] &= h'(u_{k,t-1}) \in \partial h(u_{k,t-1}), \\
			\mathbb{E}\big[ \| H(u_{k,t-1}, \xi_{k,t}) - h'(u_{k,t-1}) \|_*^2\mid\mathcal F_{k,t-1} \big]&\leq \sigma^2,
		\end{align*}
		where $\xi_{k,t}$ is sampled independently of $\mathcal F_{k,t-1}$.
	\end{assumption}
	We also make the following assumptions on the algorithm parameters $\beta_k$ and $\gamma_k$.
	\begin{assumption}\label{ass:1}
		The parameters $\{\beta_k\}$ and $\{\gamma_k\}$ satisfy $\beta_k \geq \dfrac{L\gamma_k}{\nu}$ for all $k \geq 1$.
	\end{assumption}
	
	\begin{assumption}\label{ass:2}
		The sequences $\{\beta_k\}$, $\{\gamma_k\}$, $\{\Gamma_k\}$ and $\{W_{k,t}\}_{1 \leq t \leq T_k}$ satisfy, for all $k \geq 2$,
		\[
		\frac{\gamma_k \!\left( \dfrac{W_{k,T_k}}{1-W_{k,T_k}} (\beta_k+\mu) + \beta_k \right)}{\Gamma_k}
		\leq 
		\frac{\gamma_{k-1} (\beta_{k-1}+\mu)}{\Gamma_{k-1} (1-W_{k-1,T_{k-1}})},
		\]
		where $W_{k,t}$ and $\Gamma_k$ are defined recursively as
		\begin{align}\label{par-wg}
			W_{k,t} = \begin{cases}
				1, & t=0, \\
				(1-w_{k,t})W_{k,t-1}, & t \geq 1,
			\end{cases} \quad
			\Gamma_k = \begin{cases}
				1, & k=0, \\
				(1-\gamma_k)\Gamma_{k-1}, & k \geq 1,
			\end{cases}
		\end{align}
		with
		\begin{align}\label{par-w}
			w_{k,t} = \frac{\beta_k+\mu}{\beta_k (1+p_{k,t}) + \mu}
			= \frac{1}{1 + \dfrac{\beta_k}{\beta_k+\mu} \, p_{k,t}}.
		\end{align}
	\end{assumption}
	\begin{theorem}
		Suppose that Assumptions \ref{ass:f}--\ref{ass:2} hold. Let $\{\bar{x}_k\}$ be the sequence generated by Algorithm \ref{alg}. Then we have
		\begin{align}
			&\mathbb{E}[\Psi (\bar{x}_N) - \Psi(x^*)] \nonumber\\
			\le&  \Gamma_N  (1-\gamma_1) [ \Psi (\bar{x}_{0}) - \Psi(x^*)  ]+ \frac{\Gamma_N  \gamma_1}{\Gamma_1} \left[  (\beta_1 +\mu) \frac{W_{1,T_1}}{1-W_{1,T_1}} +  \beta_1 \right] V(x_0,x^*) \nonumber\\
			& +  \frac{M^2+\sigma^2}{\nu}\cdot \Gamma_N\sum_{k=1}^N \frac{\gamma_k W_{k,T_k}}{\Gamma_k\beta_k (1-W_{k,T_k})}\sum_{i=1}^{T_k} \left[  \frac{\beta_k+\mu}{\beta_k (1+p_{k,i})+\mu}\cdot  \frac{1}{ p_{k,i}W_{k,i}}\right].\label{keyresult}
		\end{align}
	\end{theorem}
	
	\begin{proof}{Proof}
		Define
		\begin{align*}
			\Phi_k(u) &= l_f(\underline{x}_k, u) + h(u) + \chi(u) + \beta_k V(x_{k-1}, u),\\
			l_h(u_{k,t-1}, u) &= h(u_{k,t-1}) + \langle h'(u_{k,t-1}), u-u_{k,t-1}\rangle,\\
			\tilde{l}_h(u_{k,t-1}, u) &= h(u_{k,t-1}) + \langle H(u_{k,t-1}, \xi_{k,t}), u-u_{k,t-1}\rangle .
		\end{align*}
		By Assumption \ref{ass:f} and the definition of $l_h$, we have
		$h(u_{k,t}) \le l_h(u_{k,t-1}, u_{k,t}) + M\|u_{k,t}-u_{k,t-1}\|$.
		Adding $l_f(\underline{x}_k, u_{k,t}) + \beta_k V(x_{k-1}, u_{k,t}) + \chi(u_{k,t})$ to both sides and using the definitions of $\Phi_k$ and $\tilde{l}_h$ yields
		\begin{align}
			\Phi_k(u_{k,t}) &\le l_f(\underline{x}_k, u_{k,t}) + \tilde{l}_h(u_{k,t-1}, u_{k,t})
			+ \beta_k V(x_{k-1}, u_{k,t}) + \chi(u_{k,t}) \nonumber \\
			&\quad + (M + \|\delta_{k,t}\|_*)\|u_{k,t}-u_{k,t-1}\|, \label{eq1}
		\end{align}
		where $\delta_{k,t} = H(u_{k,t-1}, \xi_{k,t}) - h'(u_{k,t-1})$.
		
		From \eqref{update-u}, the strong convexity of $\chi$ and Lemma \ref{lemma:1.1}, we obtain
		\begin{align}
			&l_f(\underline{x}_k, u_{k,t}) + \tilde{l}_h(u_{k,t-1}, u_{k,t}) + \beta_k V(x_{k-1}, u_{k,t})
			+ \chi(u_{k,t}) + \beta_k p_{k,t} V(u_{k,t-1}, u_{k,t}) \nonumber \\
			&\le l_f(\underline{x}_k, u) + \tilde{l}_h(u_{k,t-1}, u) + \beta_k V(x_{k-1}, u)
			+ \chi(u) + \beta_k p_{k,t} V(u_{k,t-1}, u) \nonumber \\
			&\quad - \beta_k(1+p_{k,t})V(u_{k,t}, u) - \frac{\mu}{2}\|u-u_{k,t}\|^2 \nonumber \\
			&\le \Phi_k(u) + \beta_k p_{k,t} V(u_{k,t-1}, u) - \beta_k(1+p_{k,t})V(u_{k,t}, u)
			- \frac{\mu}{2}\|u-u_{k,t}\|^2 + \langle\delta_{k,t}, u-u_{k,t-1}\rangle, \label{eq2}
		\end{align}
		where the last inequality uses the convexity of $h$ and the definition of $\Phi_k$.
		Combining \eqref{eq1} and \eqref{eq2} gives
		\begin{align}
			\Phi_k(u_{k,t}) &\le \Phi_k(u) + \beta_k p_{k,t} V(u_{k,t-1}, u) - \beta_k(1+p_{k,t})V(u_{k,t}, u)
			- \frac{\mu}{2}\|u-u_{k,t}\|^2 \nonumber \\
			&\quad + \langle\delta_{k,t}, u-u_{k,t-1}\rangle - \beta_k p_{k,t} V(u_{k,t-1}, u_{k,t})
			+ (M+\|\delta_{k,t}\|_*)\|u_{k,t}-u_{k,t-1}\|. \label{eq3.1}
		\end{align}
		By the strong convexity of $\omega$,
		\begin{align}
			&{-}\beta_k p_{k,t} V(u_{k,t-1}, u_{k,t}) + (M+\|\delta_{k,t}\|_*)\|u_{k,t}-u_{k,t-1}\| \nonumber \\
			&\le -\frac{\nu\beta_k p_{k,t}}{2}\|u_{k,t}-u_{k,t-1}\|^2
			+ (M+\|\delta_{k,t}\|_*)\|u_{k,t}-u_{k,t-1}\| \le \frac{(M+\|\delta_{k,t}\|_*)^2}{2\nu\beta_k p_{k,t}}, \label{eq3.2}
		\end{align}
		where the last inequality uses $-at^2/2 + bt \le b^2/(2a)$ for any $a>0$.
		From \eqref{eq3.1} and \eqref{eq3.2} we obtain
		\begin{align}
			\Phi_k(u_{k,t}) &\le \Phi_k(u) + \beta_k p_{k,t} V(u_{k,t-1}, u)
			- [\beta_k(1+p_{k,t})+\mu]V(u, u_{k,t}) \nonumber \\
			&\quad + \langle\delta_{k,t}, u-u_{k,t-1}\rangle
			+ \frac{(M+\|\delta_{k,t}\|_*)^2}{2\nu\beta_k p_{k,t}}, \label{eq3}
		\end{align}
		where we also used $V(u_{k,t},u) \le \frac12\|u-u_{k,t}\|^2$.
		\eqref{eq3} implies
		\begin{align}
			&(\beta_k+\mu)V(u_{k,t},u) + \frac{\beta_k+\mu}{\beta_k(1+p_{k,t})+\mu}\bigl[\Phi_k(u_{k,t})-\Phi_k(u)\bigr] \nonumber \\
			&\le \frac{(\beta_k+\mu)\beta_k p_{k,t}}{\beta_k(1+p_{k,t})+\mu}V(u_{k,t-1},u)
			+ \frac{\beta_k+\mu}{\beta_k(1+p_{k,t})+\mu}
			\Bigl(\langle\delta_{k,t}, u-u_{k,t-1}\rangle
			+ \frac{M^2+\|\delta_{k,t}\|_*^2}{\nu\beta_k p_{k,t}}\Bigr). \label{eq-rec}
		\end{align}
		Applying Lemma \ref{lem2} to \eqref{eq-rec} yields
		\begin{align}
			&\frac{\beta_k+\mu}{1-W_{k,T_k}}V(u_{k,T_k},u)
			+ \frac{W_{k,T_k}}{1-W_{k,T_k}}\sum_{i=1}^{T_k}
			\frac{\beta_k+\mu}{\beta_k(1+p_{k,i})+\mu}
			\frac{\Phi_k(u_{k,i})-\Phi_k(u)}{W_{k,i}} \nonumber \\
			&\le (\beta_k+\mu)\frac{W_{k,T_k}}{1-W_{k,T_k}}V(u_{k,0},u) \nonumber \\
			&\quad + \frac{W_{k,T_k}}{1-W_{k,T_k}}\sum_{i=1}^{T_k}
			\frac{\beta_k+\mu}{\beta_k(1+p_{k,i})+\mu}
			\Bigl(\frac{\langle\delta_{k,i}, u-u_{k,i-1}\rangle}{W_{k,i}}
			+ \frac{M^2+\|\delta_{k,i}\|_*^2}{\nu\beta_k p_{k,i} W_{k,i}}\Bigr). \label{eq9}
		\end{align}
		Let $\theta_{k,t} = \frac{W_{k,t-1}-W_{k,t}}{(1-W_{k,t})W_{k,t-1}}$. From the definition of $\tilde{u}_{k,t}$ in \eqref{update-tu} and $W_{k,0}=1$ we have
		\begin{align*}
			\tilde{u}_{k,t} &= \frac{W_{k,t}}{1-W_{k,t}}
			\Bigl(\frac{1-W_{k,t-1}}{W_{k,t-1}}\tilde{u}_{k,t-1} + \frac{w_{k,t}}{W_{k,t}}u_{k,t}\Bigr) = \cdots = \frac{W_{k,t}}{1-W_{k,t}}\sum_{i=1}^t \frac{w_{k,i}}{W_{k,i}}u_{k,i} .
		\end{align*}
		Using the convexity of $\Phi_k$ and the definition of $w_{k,t}$ in \eqref{par-w},
		\begin{align}
			\Phi_k(\tilde{x}_k) - \Phi_k(u)
			&= \Phi_k(\tilde{u}_{k,T_k}) - \Phi_k(u) \le \frac{W_{k,T_k}}{1-W_{k,T_k}}\sum_{i=1}^{T_k}
			\frac{\beta_k+\mu}{\beta_k(1+p_{k,i})+\mu}
			\frac{\Phi_k(u_{k,i})-\Phi_k(u)}{W_{k,i}}. \label{temp3}
		\end{align}
		Setting $u_{k,0} = x_{k-1}$ and $x_k = u_{k,T_k}$ in \eqref{eq9} and combining with \eqref{temp3} gives
		\begin{align}
			\Phi_k(\tilde{x}_k) - \Phi_k(u)
			&\le (\beta_k+\mu)\frac{W_{k,T_k}}{1-W_{k,T_k}}V(x_{k-1},u)
			- \frac{\beta_k+\mu}{1-W_{k,T_k}}V(x_k,u) \nonumber \\
			&\quad + \frac{W_{k,T_k}}{1-W_{k,T_k}}\sum_{i=1}^{T_k}
			\frac{\beta_k+\mu}{\beta_k(1+p_{k,i})+\mu}
			\Bigl(\frac{\langle\delta_{k,i}, u-u_{k,i-1}\rangle}{W_{k,i}}
			+ \frac{M^2+\|\delta_{k,i}\|_*^2}{\nu\beta_k p_{k,i} W_{k,i}}\Bigr). \label{temp4}
		\end{align}
		On the other hand, Assumption \ref{ass:f} implies
		\begin{align}
			f(\bar{x}_k) &\le l_f(\underline{x}_k, \bar{x}_k) + \frac{L}{2}\|\underline{x}_k-\bar{x}_k\|^2 \nonumber \\
			&\le (1-\gamma_k)f(\bar{x}_{k-1})
			+ \gamma_k\bigl[l_f(\underline{x}_k, \tilde{x}_k) + \beta_k V(x_{k-1}, \tilde{x}_k)\bigr] - \bigl(\gamma_k\beta_k - \frac{L\gamma_k^2}{\nu}\bigr)V(x_{k-1}, \tilde{x}_k), \label{eq4}
		\end{align}
		where the last inequality uses the strong convexity of $\omega$ and convexity of $f$.
		By convexity of $h$ and strong convexity of $\chi$,
		\begin{align}
			h(\bar{x}_k) + \chi(\bar{x}_k)
			&\le (1-\gamma_k)\bigl[h(\bar{x}_{k-1}) + \chi(\bar{x}_{k-1})\bigr]
			+ \gamma_k\bigl[h(\tilde{x}_k) + \chi(\tilde{x}_k)\bigr] \nonumber \\
			&\quad - \frac{\mu}{2}\gamma_k(1-\gamma_k)\|\bar{x}_{k-1} - \tilde{x}_k\|^2. \label{eq5}
		\end{align}
		Adding \eqref{eq4} and \eqref{eq5} and using the definition of $\Psi$,
		\begin{align}
			\Psi(\bar{x}_k) &\le (1-\gamma_k)\Psi(\bar{x}_{k-1}) + \gamma_k\Phi_k(\tilde{x}_k) \nonumber \\
			&\quad - \frac{\mu}{2}\gamma_k(1-\gamma_k)\|\bar{x}_{k-1} - \tilde{x}_k\|^2
			- \bigl(\gamma_k\beta_k - \frac{L\gamma_k^2}{\nu}\bigr)V(x_{k-1}, \tilde{x}_k). \nonumber
		\end{align}
		Hence,
		\begin{align}
			\Psi(\bar{x}_k) - \Psi(u)
			&\le (1-\gamma_k)\bigl[\Psi(\bar{x}_{k-1}) - \Psi(u)\bigr]
			+ \gamma_k\bigl[\Phi_k(\tilde{x}_k) - \Psi(u)\bigr] \nonumber \\
			&\quad - \frac{\mu}{2}\gamma_k(1-\gamma_k)\|\bar{x}_{k-1} - \tilde{x}_k\|^2
			- \bigl(\gamma_k\beta_k - \frac{L\gamma_k^2}{\nu}\bigr)V(x_{k-1}, \tilde{x}_k). \label{temp1}
		\end{align}
		By definition of $\Phi_k$, for any $u\in X$,
		\begin{align}
			\Phi_k(u) &\le f(u) + h(u) + \chi(u) + \beta_k V(x_{k-1}, u) = \Psi(u) + \beta_k V(x_{k-1}, u). \label{temp2}
		\end{align}
		Substituting \eqref{temp2} into \eqref{temp1} yields
		\begin{align}
			\Psi(\bar{x}_k) - \Psi(u)
			&\le (1-\gamma_k)\bigl[\Psi(\bar{x}_{k-1}) - \Psi(u)\bigr]
			+ \gamma_k\bigl[\Phi_k(\tilde{x}_k) - \Phi_k(u) + \beta_k V(x_{k-1}, u)\bigr] \nonumber \\
			&\quad - \bigl(\gamma_k\beta_k - \frac{L\gamma_k^2}{\nu}\bigr)V(x_{k-1}, \tilde{x}_k)
			- \frac{\mu}{2}\gamma_k(1-\gamma_k)\|\bar{x}_{k-1} - \tilde{x}_k\|^2. \label{eq7}
		\end{align}
		Using Assumption \ref{ass:1} we obtain
		\begin{align}
			\Psi(\bar{x}_k) - \Psi(u)
			&\le (1-\gamma_k)\bigl[\Psi(\bar{x}_{k-1}) - \Psi(u)\bigr]
			+ \gamma_k\bigl[\Phi_k(\tilde{x}_k) - \Phi_k(u) + \beta_k V(x_{k-1}, u)\bigr]. \label{eq8}
		\end{align}
		Now combine \eqref{temp4} and \eqref{eq8}:
		\begin{align}
			\Psi(\bar{x}_k) - \Psi(u)
			&\le (1-\gamma_k)\bigl[\Psi(\bar{x}_{k-1}) - \Psi(u)\bigr] \nonumber \\
			&\quad + \gamma_k\Bigl[
			\Bigl(\frac{(\beta_k+\mu)W_{k,T_k}}{1-W_{k,T_k}} + \beta_k\Bigr)V(x_{k-1}, u)
			- \frac{\beta_k+\mu}{1-W_{k,T_k}}V(x_k,u)
			\Bigr] \nonumber \\
			&\quad + \gamma_k\frac{W_{k,T_k}}{1-W_{k,T_k}}
			\sum_{i=1}^{T_k}\frac{\beta_k+\mu}{\beta_k(1+p_{k,i})+\mu}
			\Bigl(\frac{\langle\delta_{k,i}, u-u_{k,i-1}\rangle}{W_{k,i}}
			+ \frac{M^2+\|\delta_{k,i}\|_*^2}{\nu\beta_k p_{k,i} W_{k,i}}\Bigr). \label{eq10}
		\end{align}
		Applying Lemma \ref{lem2} to \eqref{eq10} gives
		\begin{align}
			&\Psi(\bar{x}_N) - \Psi(u) \nonumber \\
			&\le \Gamma_N (1-\gamma_1)\bigl[\Psi(\bar{x}_{0}) - \Psi(u)\bigr] \nonumber \\
			&\quad + \Gamma_N\sum_{k=1}^N \frac{\gamma_k}{\Gamma_k}
			\Bigl[\Bigl(\frac{(\beta_k+\mu)W_{k,T_k}}{1-W_{k,T_k}} + \beta_k\Bigr)V(x_{k-1}, u)
			- \frac{\beta_k+\mu}{1-W_{k,T_k}}V(x_k,u)\Bigr] \nonumber \\
			&\quad + \Gamma_N\sum_{k=1}^N \frac{\gamma_k W_{k,T_k}}{\Gamma_k(1-W_{k,T_k})}
			\sum_{i=1}^{T_k}\frac{\beta_k+\mu}{\beta_k(1+p_{k,i})+\mu}
			\Bigl(\frac{\langle\delta_{k,i}, u-u_{k,i-1}\rangle}{W_{k,i}}
			+ \frac{M^2+\|\delta_{k,i}\|_*^2}{\nu\beta_k p_{k,i} W_{k,i}}\Bigr). \label{eq11}
		\end{align}
		Under Assumption \ref{ass:2}, inequality \eqref{eq11} simplifies to
		\begin{align}
			&\Psi(\bar{x}_N) - \Psi(u) \nonumber \\
			&\le \Gamma_N (1-\gamma_1)\bigl[\Psi(\bar{x}_{0}) - \Psi(u)\bigr]
			+ \frac{\Gamma_N\gamma_1}{\Gamma_1}
			\Bigl[(\beta_1+\mu)\frac{W_{1,T_1}}{1-W_{1,T_1}} + \beta_1\Bigr]V(x_0,u) \nonumber \\
			&\quad + \Gamma_N\sum_{k=1}^N \frac{\gamma_k W_{k,T_k}}{\Gamma_k(1-W_{k,T_k})}
			\sum_{i=1}^{T_k}\frac{\beta_k+\mu}{\beta_k(1+p_{k,i})+\mu}
			\Bigl(\frac{\langle\delta_{k,i}, u-u_{k,i-1}\rangle}{W_{k,i}}
			+ \frac{M^2+\|\delta_{k,i}\|_*^2}{\nu\beta_k p_{k,i} W_{k,i}}\Bigr). \label{eq12}
		\end{align}
		By Assumption~\ref{ass:h}, $u_{k,i-1}$ is $\mathcal F_{k,i-1}$-measurable and
		\begin{align*}
			\mathbb E[\langle\delta_{k,i},x^*-u_{k,i-1}\rangle\mid\mathcal F_{k,i-1}]
			&=\langle\mathbb E[\delta_{k,i}\mid\mathcal F_{k,i-1}],x^*-u_{k,i-1}\rangle=0.
		\end{align*}
		Setting $u=x^*$ and taking expectations in \eqref{eq12} therefore yields \eqref{keyresult}.
		\Halmos
	\end{proof}
	\begin{theorem}\label{thm2}
		Suppose that Assumptions \ref{ass:f}--\ref{ass:2} hold. Let $\{\bar{x}_k\}$ be the sequence generated by Algorithm \ref{alg}. Set
		\begin{align}
			\beta_k &=\beta=\frac{L(1-c)}{\nu},\quad \gamma_k=\gamma=1-c,\quad T_k =\left\lceil \frac{1}{c^{k/2}}\cdot \frac{(\beta+\mu)(1-c)}{ c(\beta+\mu) -\beta}\right\rceil,\label{para-beta}\\
			N&= \max\left\{1,\left\lceil\frac{2 \log \max\{\frac{A}{\epsilon},1\}}{\log (\frac{1}{c})}\right\rceil\right\},\quad p_{k,t} = \frac{\beta+\mu}{\beta} \cdot  \frac{1}{c^{k/2}}, \quad\theta_{k,t} = \frac{1- \frac{1}{1+c^{k/2}} }{1- \frac{1}{(1+c^{k/2})^t} }.\label{para-N}
		\end{align}
		where $c = \dfrac{\sqrt{\frac{L}{\mu\nu}}}{1+\sqrt{\frac{L}{\mu\nu}}}$.
		Then we have
		\begin{align}\label{thm:1}
			\mathbb{E}\bigl[\Psi(\bar{x}_N) - \Psi(x^*)\bigr] \leq c^{N/2} \cdot A,
		\end{align}
		where $A = \Delta_0 + (\beta+\mu)(1-c) V(x_0,x^*) + \dfrac{2(M^2+\sigma^2)}{\nu (\beta+\mu)}$ and $\Delta_0 = \Psi(\bar{x}_0) - \Psi(x^*)$.
		Consequently, to obtain an $\epsilon$-solution, the total number of evaluations for $\nabla f$ and $h'$, respectively, are bounded by
		\begin{align}\label{thm:2}
			\mathcal{O}\left(1+ \sqrt{\frac{L}{\nu\mu}} \,
			\log \max\!\Bigl\{1, \frac{A}{\epsilon}\Bigr\} \right)\quad \text{and}\quad \mathcal{O}\left( \frac{AL}{\nu\mu\epsilon} + 
			\sqrt{\frac{L}{\nu\mu}} \,
			\log \max\!\Bigl\{1, \frac{A}{\epsilon}\Bigr\} \right).
		\end{align}
	\end{theorem}
	\begin{proof}{Proof}
		From \eqref{para-beta} we have $\beta = L\gamma/\nu$, which immediately satisfies Assumption \ref{ass:1}. Moreover, a simple calculation gives
		$\beta \le \frac{c}{1-c}\,\mu$,
		which is equivalent to
		$c(\beta+\mu) - \beta > 0. $
		Hence $T_k$ is well defined.
		Let $c_2^{(k)} = \dfrac{1}{1+c^{k/2}}$. Setting $p_{k,t} = \dfrac{\beta+\mu}{\beta}\cdot\dfrac{c_2^{(k)}}{1-c_2^{(k)}}$ and 
		$\tilde{p}_{k,t} = \dfrac{\beta}{\beta+\mu}p_{k,t} = \dfrac{c_2^{(k)}}{1-c_2^{(k)}}$, we obtain from \eqref{par-w} that $w_{k,t} = \dfrac{1}{1+\tilde{p}_{k,t}}$ and consequently $W_{k,t} = (c_2^{(k)})^t$. Using the definition of $T_k$, we have
		$$
		W_{k,T_k} = \Bigl(\frac{1}{1+c^{k/2}}\Bigr)^{T_k}
		\le \frac{c(\beta+\mu)-\beta}{\mu+c(\beta+\mu)}.
		$$
		Therefore,
		\begin{align}
			\frac{(\beta+\mu)W_{k,T_k} }{\bigl(1-W_{k,T_k}\bigr) }+ \beta\le c(\beta+\mu), \label{thm2:Wk}
		\end{align}
		and Assumption \ref{ass:2} is satisfied.
		Substituting the constant parameters into \eqref{keyresult} yields
		\begin{align}
			&\mathbb{E}\bigl[\Psi(\bar{x}_N) - \Psi(x^*)\bigr] \nonumber \\
			&\le \Gamma_N (1-\gamma)\bigl[\Psi(\bar{x}_0) - \Psi(x^*)\bigr]
			+ \frac{\Gamma_N\gamma}{\Gamma_1}
			\Bigl[(\beta+\mu)\frac{W_{1,T_1}}{1-W_{1,T_1}} + \beta\Bigr]V(x_0,x^*) \nonumber \\
			&\quad + \frac{M^2+\sigma^2}{\nu}\,
			\Gamma_N\sum_{k=1}^N \frac{\gamma}{\Gamma_k(\beta+\mu)}
			\sum_{i=1}^{T_k}
			\Bigl[\frac{W_{k,T_k}}{(1+\tilde{p}_{k,i})\tilde{p}_{k,i} W_{k,i} (1-W_{k,T_k})}\Bigr]. \label{key}
		\end{align}
		First, from the definitions of $\gamma$ and $\Gamma_k$ we obtain $\Gamma_k = c^k$. Setting $\Delta_0 = \Psi(\bar{x}_0) - \Psi(x^*)$, we have
		\begin{align}
			\Gamma_N (1-\gamma)\bigl[\Psi(\bar{x}_0) - \Psi(x^*)\bigr] = c^N \cdot c\Delta_0. \label{TermI}
		\end{align}
		Second, using \eqref{thm2:Wk} together with the definitions of $\gamma$ and $\Gamma_k$ yields
		\begin{align}
			\frac{\Gamma_N \gamma}{\Gamma_1}
			\Bigl[(\beta+\mu)\frac{W_{1,T_1}}{1-W_{1,T_1}} + \beta\Bigr]V(x_0,x^*)
			\leq c^N \cdot (1-c)(\beta+\mu) V(x_0,x^*). \label{TermII}
		\end{align}
		Finally, substituting $W_{k,i}=(c_2^{(k)})^i$ and $\widetilde {p}_{k,i}=c_2^{(k)}/(1-c_2^{(k)})$ and summing the resulting geometric series gives
		\[
		\sum_{i=1}^{T_k} 
		\frac{W_{k,T_k}}{(1+\tilde{p}_{k,i})\tilde{p}_{k,i} W_{k,i} (1-W_{k,T_k})}
		= \frac{1-c_2^{(k)}}{c_2^{(k)}} = c^{k/2},
		\]
		and consequently
		\begin{align}
			&\quad\frac{M^2+\sigma^2}{\nu}\cdot \Gamma_N\sum_{k=1}^N \frac{\gamma_k}{\Gamma_k(\beta_k+\mu)}\sum_{i=1}^{T_k} \left[ \frac{ W_{k,T_k}}{(1+\tilde{p}_{k,i}) \tilde{p}_{k,i}W_{k,i} (1-W_{k,T_k})}\right] \nonumber\\
			&= \frac{M^2+\sigma^2}{\nu}\cdot c^N \sum_{k=1}^N \frac{1-c}{c^k (\beta+\mu)}\cdot c^{k/2}  \le c^{N/2}\cdot \frac{M^2+\sigma^2}{\nu (\beta+\mu)} (1+\sqrt{c}).\label{TermIII}
		\end{align}
		Combining \eqref{key}, \eqref{TermI}, \eqref{TermII} and \eqref{TermIII}, we obtain
		\begin{align}
			\mathbb{E}\bigl[\Psi(\bar{x}_N) - \Psi(x^*)\bigr]
			&\le c^N \cdot c\Delta_0
			+ c^N \cdot (1-c)(\beta+\mu) V(x_0,x^*) + c^{N/2}\cdot 
			\frac{M^2+\sigma^2}{\nu (\beta+\mu)} (1+\sqrt{c}). \label{combined}
		\end{align}
		Define
		$
		A = \Delta_0 + (\beta+\mu)(1-c) V(x_0,x^*) + \frac{2(M^2+\sigma^2)}{\nu (\beta+\mu)}.
		$
		Since $c<1$, we obtain from \eqref{combined} the simplified bound
		$
		\mathbb{E}\bigl[\Psi(\bar{x}_N) - \Psi(x^*)\bigr] \le c^{N/2}\cdot A.
		$
		Now choose
		\[
		N = \max\left\{1,\left\lceil\frac{2\log \max\{\frac{A}{\epsilon},1\}}{\log\!\bigl(\frac{1}{c}\bigr)}\right\rceil\right\}
		= \max\left\{1,\left\lceil\frac{2\log \max\{\frac{A}{\epsilon},1\}}{\log\!\bigl(1+\sqrt{\frac{\mu\nu}{L}}\bigr)}\right\rceil\right\},
		\]
		which is of order 
		$\mathcal{O}\!\left(1+\sqrt{\frac{L}{\nu\mu}}\log\max\!\bigl\{1,\frac{A}{\epsilon}\bigr\}\right)$.
		With this $N$, we have $\mathbb{E}[\Psi(\bar{x}_N) - \Psi(x^*)] \le \epsilon$.
		Moreover, the total number of inner iterations can be bounded as follows:
		\begin{align*}
			\sum_{k=1}^N T_k
			&\le \sum_{k=1}^N\Bigl[
			\frac{1}{c^{k/2}}\cdot\frac{(\beta+\mu)(1-c)}{c(\beta+\mu)-\beta} + 1
			\Bigr] \le \frac{(\beta+\mu)(1-c)}{c(\beta+\mu)-\beta}\,
			\sum_{k=1}^N \frac{1}{c^{k/2}} \;+\; N \\
			&\le \frac{(\beta+\mu)(1-c)}{c(\beta+\mu)-\beta}\,
			\frac{1}{c^{N/2}(1-\sqrt{c})} \;+\; N = \mathcal{O}\!\left(
			\frac{AL}{\nu\mu\epsilon} + 
			\sqrt{\frac{L}{\nu\mu}}\log\max\!\Bigl\{1,\frac{A}{\epsilon}\Bigr\}
			\right).
		\end{align*}
		This completes the proof.
		\Halmos
	\end{proof}
	\begin{remark}
		When exact subgradients of $h$ are available (i.e., $\sigma = 0$), Theorem~\ref{thm2} shows that an $\epsilon$-solution can be found using at most
		$\mathcal{O}\!\left(1+\sqrt{\frac{L}{\mu}}\,\log\max\!\Bigl\{1,\frac{A}{\epsilon}\Bigr\}\right)$
		evaluations of $\nabla f$, and at most
		$\mathcal{O}\!\left(\frac{AL}{\mu\epsilon}+\sqrt{\frac{L}{\mu}}\,\log\max\!\Bigl\{1,\frac{A}{\epsilon}\Bigr\}\right)$
		evaluations of $h'$.
	\end{remark}

\begin{remark}\label{remark:2}
	Suppose that $f$ is $L$-smooth and $\mu$-strongly convex with respect to the Euclidean norm, and that $\chi$ is a relatively simple convex function. Fix any $x_c\in X$ and define
	\[
	\widetilde f(x):=f(x)-\mu V(x_c,x),
	\qquad
	\widetilde\chi(x):=\chi(x)+\mu V(x_c,x).
	\]
	Since the two additional terms cancel, problem~\eqref{prob} admits the equivalent representation
	\[
	\min_{x\in X}
	\bigl\{\widetilde f(x)+h(x)+\widetilde\chi(x)\bigr\}.
	\]
	This reformulation transfers the strong convexity of $f$ to the relatively simple term $\widetilde\chi$, while preserving the original objective function.
	
	Indeed, for any $x,z\in X$, we have
	\begin{align*}
		&\widetilde f(z)-\widetilde f(x)
		-\langle\nabla\widetilde f(x),z-x\rangle =
		f(z)-f(x)-\langle\nabla f(x),z-x\rangle
		-\mu V(x,z).
	\end{align*}
	The smoothness and strong convexity of $f$, together with the bounds on the Bregman distance, yield
	\[
	\frac{\mu}{2}\|z-x\|^2
	\le
	f(z)-f(x)-\langle\nabla f(x),z-x\rangle
	\le
	\frac{L}{2}\|z-x\|^2
	\]
	and
	\[
	\frac{\nu}{2}\|z-x\|^2
	\le
	V(x,z)
	\le
	\frac{1}{2}\|z-x\|^2.
	\]
	Consequently,
	\[
	0\le
	\widetilde f(z)-\widetilde f(x)
	-\langle\nabla\widetilde f(x),z-x\rangle
	\le
	\frac{L-\mu\nu}{2}\|z-x\|^2.
	\]
	Therefore, $\widetilde f$ is convex and $(L-\mu\nu)$-smooth.
	
	Moreover, for any $\chi'(x)\in\partial\chi(x)$, define
$\widetilde\chi'(x):=\chi'(x)+\mu\bigl(\nabla\omega(x)-\nabla\omega(x_c)\bigr).$
	By the convexity of $\chi$ and the Bregman three-point identity,
	\begin{align*}
		\widetilde\chi(z)-\widetilde\chi(x)
		-\langle\widetilde\chi'(x),z-x\rangle & =
		\chi(z)-\chi(x)-\langle\chi'(x),z-x\rangle
		+\mu V(x,z) \\
		& \geq
		\mu V(x,z)
		\geq
		\frac{\mu\nu}{2}\|z-x\|^2.
	\end{align*}
	Thus, $\widetilde\chi$ is $\widetilde\mu$-strongly convex, where
	$\widetilde\mu:=\mu\nu$. Since the additional Bregman term is included in the proximal component, $\widetilde\chi$ retains the relatively simple structure required by the proposed methods.
	
	This observation shows that the assumption that strong convexity is carried by $\chi$ is not restrictive. In particular, whenever strong convexity is originally present in the smooth component $f$, it can be incorporated into the proximal term through the above equivalent reformulation, without changing the optimization problem. Hence, Theorem~\ref{thm2} applies with the smoothness and strong-convexity parameters replaced by $L-\mu\nu$ and $\mu\nu$, respectively. This equivalence also enables the proposed restart-free methods to cover problem classes in which strong convexity is carried by $f$, including the setting considered in \cite{Lan2022}, 
	while 
	avoiding explicit restart or multistage mechanisms.
\end{remark}

	\section{Nonsmooth Problems with Computable Smoothing Families}\label{sec:3}
	
	In this section, we study the setting in which directly computing a subgradient of $h$ is difficult, while a suitable smooth approximation of $h$ and its first-order information are available. This setting is relevant to a broad class of max-form nonsmooth functions, for which smoothing enables the use of gradient-based techniques. We impose the following assumption on the smoothing family.
	
	\begin{assumption}\label{ass:smoothing}
		The convex function $h:X\to\mathbb R$ admits a family of continuously differentiable convex functions $\{h_\eta\}_{\eta>0}$. There exist constants $B_h,C_h>0$ such that, for all $x,z\in X$ and $0<\eta'\le\eta$,
		\begin{align}
			h_\eta(x)&\le h(x)\le h_\eta(x)+B_h\eta,
			\qquad
			\|\nabla h_\eta(x)-\nabla h_\eta(z)\|_*
			\le\frac{C_h}{\eta}\|x-z\|,\label{rem:smoothing-family}\\
			0&\le h_{\eta'}(x)-h_\eta(x)
			\le B_h(\eta-\eta').\label{rem:smoothing-cross}
		\end{align}
		For every $\eta>0$, the value $h_\eta(x)$ can be computed, and first-order information for $h_\eta$ is available either exactly or through a stochastic first-order oracle.
	\end{assumption}
	
	Set $L_\eta:=C_h/\eta$. The second inequality in \eqref{rem:smoothing-family} shows that $L_\eta$ is a Lipschitz constant of $\nabla h_\eta$. For a fixed smoothing parameter $\eta>0$, consider the smooth composite approximation
	\begin{align}\label{prob:smooth}
		\min_{x\in X}\;\phi_\eta(x):=f(x)+h_\eta(x)+\chi(x).
	\end{align}
	By the first inequality in \eqref{rem:smoothing-family}, the approximate objective $\phi_\eta$ provides a controlled approximation to the original objective $\Psi$:
	\begin{align}\label{smooth:transfer}
		\phi_\eta(x)\le\Psi(x)\le\phi_\eta(x)+B_h\eta,\qquad x\in X.
	\end{align}
	Thus, the smoothing error is explicitly controlled by the parameter $\eta$. In particular, choosing $\eta=\epsilon/(2B_h)$ implies that any $(\epsilon/2)$-solution of \eqref{prob:smooth} is an $\epsilon$-solution of the original problem~\eqref{prob}. This reduction allows us to apply an accelerated stochastic smoothed-gradient sliding method to the approximation while progressively decreasing the smoothing parameter, thereby avoiding the explicit multistage restart mechanism used in existing approaches.

	\subsection{The Restart‑Free Accelerated Stochastic Smoothed Gradient Sliding Algorithm}
	In this subsection, we propose a restart‑free accelerated stochastic smoothed gradient sliding (RF‑ASSGS) algorithm for solving problem \eqref{prob:smooth}. 
	As in Section~\ref{sec2}, we assume that first‑order information for $h_\eta$ is obtained via a stochastic oracle. At iteration $t$, given an input point $x_{k,t}\in X$, the oracle returns a vector $H_\eta(x_{k,t},\xi_{k,t})$, where $\xi_{k,t}$ is a fresh random sample. 
	
	The RF‑ASSGS algorithm proceeds in two nested loops per outer iteration. In the outer loop, we perform an accelerated proximal gradient step:
	\begin{align}
		\underline{x}_k &= (1-\gamma_k)\bar{x}_{k-1} + \gamma_k x_{k-1}, \label{up2-uxk}\\
		\tilde{x}_k &\approx \arg\min_{u\in X}\Big\{
		l_f(\underline{x}_k,u) + h_{\eta_k}(u) +\chi(u)+ \beta_k V(x_{k-1},u)
		\Big\}, \label{up2-xk}\\
		\bar{x}_k &= (1-\lambda_k)\bar{x}_{k-1} + \lambda_k \tilde{x}_k, \label{up2-bxk}
	\end{align}
	where $l_f(\underline{x}_k,u)=f(\underline{x}_k)+\langle\nabla f(\underline{x}_k),u-\underline{x}_k\rangle$ is the linearization of $f$, and $V(\cdot,\cdot)$ denotes the Bregman distance.
	
	In the inner loop, the gradient $\nabla f(\underline{x}_k)$ is reused throughout $T_k$ inner updates. The subproblem \eqref{up2-xk} is solved approximately by the following accelerated stochastic gradient steps:
	\begin{align*}
		\underline{u}_{k,t} &= (1-\lambda_k)\bar{x}_{k-1}
		+ \lambda_k(1-\alpha_{k,t})\tilde{u}_{k,t-1}
		+ \lambda_k\alpha_{k,t} u_{k,t-1},\\
		u_{k,t}&=\arg \min _{u \in X}\left\{l_f\left(\underline{x}_k, u\right)+\tilde{l}_{h_{\eta_k}}\left(\underline{u}_{k,t}, u\right)+\chi(u)+\beta_k V\left(x_{k-1}, u\right) \right.\nonumber\\
		&\qquad\qquad\left. +\left(\beta_k p_{k,t}+q_{k,t}\right) V\left(u_{k,t-1}, u\right)\right\},\\
		\tilde{u}_{k,t} &= (1-\alpha_{k,t})\tilde{u}_{k,t-1} + \alpha_{k,t} u_{k,t},
	\end{align*}
	where $\tilde{l}_{h_{\eta_k}}(\underline{u}_{k,t},u)=h_{\eta_k}(\underline{u}_{k,t})+\langle H_{\eta_k}(\underline{u}_{k,t},\xi_{k,t}),u-\underline{u}_{k,t}\rangle$. The complete RF-ASSGS algorithm is formally described in Algorithm \ref{alg2}. 
	\begin{algorithm}[t]
		\caption{Restart-free accelerated stochastic smoothed-gradient sliding (RF-ASSGS) algorithm}
		\label{alg2}
		\begin{algorithmic}[1]
			\REQUIRE $x_0\in X$, $\epsilon\in(0,8\Delta_0)$, a certified bound $\Delta_0>0$, $B_h,C_h>0$, and $\sigma\ge0$.
			\STATE Set $\bar{x}_0=x_0$ and\par\noindent
			\begin{minipage}{\linewidth}
				\begin{align*}
					\widehat L&=\max\{L,4\nu\mu\},\quad\lambda_k=\lambda = \sqrt{\frac{\nu \mu}{\widehat L}},\quad \gamma_k=\gamma =\frac{\lambda}{2}, \quad\beta_k =\beta=\frac{\widehat L\gamma}{\nu},\quad
					N=\left\lceil\frac{2\log(8\Delta_0/\epsilon)}{-\log(1-\gamma)}\right\rceil.
				\end{align*}
			\end{minipage}
			\FOR{$k=1,2,\cdots, N$}
			\STATE Set\par\noindent
			\begin{minipage}{\linewidth}
				\begin{align*}
					\eta_k&=\frac{\Delta_0}{B_h}(1-\gamma)^{k/2},\quad
					L_{\eta_k}=\frac{C_h}{\eta_k},\quad T_k^{\rm s}
					=
					\left\lceil
					\frac{8\sigma^2\log 2(1-(1-\gamma)^{N/2})}{5\widehat L\epsilon(1-\sqrt{1-\gamma})}
					(1-\gamma)^{(N-k)/2}
					\right\rceil,\nonumber\\
					T_k^{\rm g}&=\left\lceil\frac{\log(1/2)}{\log\bigl(1-\frac{1}4\sqrt{\frac{\widehat L}{\max\{\widehat L,L_{\eta_k}\}}}\bigr)}\right\rceil,\quad T_k=\max\{T_k^{\rm g},T_k^{\rm s}\},
					\quad
					\alpha_{k,t}=\alpha_k=1-2^{-1/T_k}.
				\end{align*}
			\end{minipage}
			\STATE Compute $\underline{x}_k = (1-\gamma_k) \bar{x}_{k-1} + \gamma_k x_{k-1}$.\\
			\STATE Set $u_{k,0}=x_{k-1}$, $\tilde{u}_{k,0} =\bar{x}_{k-1} $.
			\FOR{$t=1,2,\cdots, T_k$}
			\STATE Set $p_{k,t}=p_k = \frac{1-\alpha_k}{\alpha_k}$, $\alpha_{k,t}=\alpha_k$, $q_{k,t}=p_k\mu$ and $\Lambda_{k,t}=\left\{\begin{array}{ll}
				1, & t=1, \\
				(1-\alpha_{k,t})\Lambda_{k,t-1}, & t> 1.
			\end{array}\right.$
			\STATE Compute\par\noindent
			\begin{minipage}{\linewidth}
				\begin{align}
					\underline{u}_{k,t}&=\left(1-\lambda_k\right) \bar{x}_{k-1}+\lambda_k\left(1-\alpha_{k,t}\right) \tilde{u}_{k,t-1}+\lambda_k \alpha_{k,t} u_{k,t-1},\label{update2-uu} \\
					u_{k,t}&=\arg \min _{u \in X}\left\{l_f\left(\underline{x}_k, u\right)+\tilde{l}_{h_{\eta_k}}\left(\underline{u}_{k,t}, u\right)+\chi(u)+\beta_k V\left(x_{k-1}, u\right) \right.\nonumber\\
					&\qquad\qquad\left. +\left(\beta_k p_{k,t}+q_{k,t}\right) V\left(u_{k,t-1}, u\right)\right\},\label{update2-u}\\
					\tilde{u}_{k,t}&=\left(1-\alpha_{k,t}\right) \tilde{u}_{k,t-1}+\alpha_{k,t} u_{k,t}.\label{update2-tu}
				\end{align}
			\end{minipage}
			\ENDFOR
			\STATE 
			Set $x_k= u_{k,T_k}$, $\tilde{x}_k = \tilde{u}_{k,T_k}$.\\
			\STATE Compute $\bar{x}_k  =  (1-\lambda_k) \bar{x}_{k-1} + \lambda_k\tilde{x}_{k}$.
			\ENDFOR
		\end{algorithmic}
	\end{algorithm}
	
	\subsection{Complexity analysis}\label{sec3.1}
	In this subsection, we establish the convergence of the RF-ASSGS algorithm for solving problem \eqref{prob} under Assumption~\ref{ass:smoothing}. We begin by stating several mild assumptions required for the analysis.
	
	\begin{assumption}\label{ass3:f}
		The function $f: X \rightarrow \mathbb{R}$ is smooth and convex. It satisfies, for some $L>0$,
		\begin{align*}
			f(x) &\leq f(y) + \langle \nabla f(y), x-y \rangle + \frac{L}{2} \|x-y\|^2,
		\end{align*}
		for all $x, y \in X$.
	\end{assumption}
	
	\begin{assumption}\label{ass:h2}
		Let $\mathcal F_{k,t-1}$ denote the $\sigma$-field generated by all oracle samples drawn before $\xi_{k,t}$. The points $\underline u_{k,t}$ and $u_{k,t-1}$ are $\mathcal F_{k,t-1}$-measurable, and there exists a constant $\sigma\ge0$, uniform over the smoothing schedule, such that
		\begin{align*}
			\mathbb{E}\bigl[ H_{\eta_k}(\underline{u}_{k,t}, \xi_{k,t})\mid\mathcal F_{k,t-1} \bigr] &= \nabla h_{\eta_k}(\underline{u}_{k,t}), \\
			\mathbb{E}\bigl[ \| H_{\eta_k}(\underline{u}_{k,t}, \xi_{k,t}) - \nabla h_{\eta_k}(\underline{u}_{k,t}) \|_*^2\mid\mathcal F_{k,t-1} \bigr] &\leq \sigma^2,
		\end{align*}
		where $\xi_{k,t}$ is sampled independently of $\mathcal F_{k,t-1}$.
	\end{assumption}
	
	In the convergence analysis, we evaluate the quality of the solution obtained from the $k$-th call to the inner procedure using the following error measure, which is also employed in \cite{Lan2022}:
	\begin{align}\label{def:Q}
		\phi_k(x):=f(x)+h_{\eta_k}(x)+\chi(x),~
		Q_k(x,u):=g_k(x)-g_k(u)+h_{\eta_k}(x)-h_{\eta_k}(u),
	\end{align}
	where $g_k(x) = \langle \nabla f(\underline{x}_k), x \rangle$.
	
	\begin{lemma}\label{lem3.2}
		Suppose Assumption \ref{ass3:f} and Assumption~\ref{ass:smoothing} hold. For any $u \in X$, we have
		\begin{align}
			\phi_k\left(\bar{x}_k\right)-\phi_k(u) \leq&(1  \left.-\gamma_k\right)\left[\phi_k\left(\bar{x}_{k-1}\right)-\phi_k(u)\right]+Q_k\left(\bar{x}_k, u\right)-\left(1-\gamma_k\right) Q_k\left(\bar{x}_{k-1}, u\right)\nonumber\\
			& +\chi(\bar{x}_k)-(1-\gamma_k)\chi(\bar{x}_{k-1})-\gamma_k\chi(u)+\frac{L}{2}\left\|\bar{x}_k-\underline{x}_k\right\|^2.\label{lem3.2:1}
		\end{align}
	\end{lemma}
	\begin{proof}{Proof}
		By Assumption \ref{ass3:f}, definition of $\phi_k$, and the convexity of $f(\cdot)$, we have
		\begin{align*}
			& \phi_k\left(\bar{x}_k\right)-\left(1-\gamma_k\right) \phi_k\left(\bar{x}_{k-1}\right)-\gamma_k \phi_k(u) \nonumber\\
			&\leq l_f\left(\underline{x}_k, \bar{x}_k\right)+\frac{L}{2}\left\|\bar{x}_k-\underline{x}_k\right\|^2+h_{\eta_k}\left(\bar{x}_k\right)-\left(1-\gamma_k\right) l_f\left(\underline{x}_k, \bar{x}_{k-1}\right) \nonumber\\
			& \quad
			-\left(1-\gamma_k\right) h_{\eta_k}\left(\bar{x}_{k-1}\right)-\gamma_k l_f\left(\underline{x}_k, u\right)-\gamma_k h_{\eta_k}(u) +\chi(\bar{x}_k)-(1-\gamma_k)\chi(\bar{x}_{k-1})-\gamma_k\chi(u)\nonumber\\
			& =Q_k\left(\bar{x}_k, u\right)-\left(1-\gamma_k\right) Q_k\left(\bar{x}_{k-1}, u\right)+\frac{L}{2}\left\|\bar{x}_k-\underline{x}_k\right\|^2+\chi(\bar{x}_k)-(1-\gamma_k)\chi(\bar{x}_{k-1})-\gamma_k\chi(u).
		\end{align*}
		The proof is completed.
		\Halmos
	\end{proof}
	Next, we derive an upper bound for the quantity
	$
	Q_k(\bar{x}_k, u) - (1-\gamma_k)Q_k(\bar{x}_{k-1}, u).
	$
	
	\begin{lemma}\label{lem3.3}
		Suppose Assumptions \ref{ass3:f}, \ref{ass:chi}, and \ref{ass:smoothing} hold. Consider the $k$-th call to the inner procedure in Algorithm~\ref{alg2} and let
		\begin{align}\label{par3:1}
			\Lambda_{k,t} = \begin{cases}
				1, & t=1, \\
				(1-\alpha_{k,t})\Lambda_{k,t-1}, & t>1,
			\end{cases}
		\end{align}
		If the parameters satisfy 
		\begin{align}
			\gamma_k&<\lambda_k \le 1,\quad
			\Lambda_{k,T_k}(1-\alpha_{k,1})=1-\frac{\gamma_k}{\lambda_k},\quad a_{k,t}=\beta_kp_{k,t}+q_{k,t},\quad
			r_{k,t}=a_{k,t}-\frac{\lambda_kL_{\eta_k}\alpha_{k,t}}{\nu}>0.\label{par3:3}
		\end{align}
		then, for all $u \in X$,
		\begin{align}\label{lem3.3:1}
			&Q_k(\bar{x}_k, u) - (1-\gamma_k)Q_k(\bar{x}_{k-1}, u)+\chi(\bar{x}_k)-(1-\gamma_k)\chi(\bar{x}_{k-1})-\gamma_k\chi(u)\nonumber\\
			\le &\Lambda_{k,T_k} \sum_{t=1}^{T_k} \frac{Y_{k,t}(u)}{\Lambda_{k,t}}
			+\Lambda_{k,T_k}\sum_{t=1}^{T_k}\frac{\lambda_k\alpha_{k,t}}{\Lambda_{k,t}}
			\left[\langle\delta_{k,t},u-u_{k,t-1}\rangle+
			\frac{\|\delta_{k,t}\|_*^2}{2\nu r_{k,t}}\right],
		\end{align}
		where $\delta_{k,t}:=H_{\eta_k}(\underline u_{k,t},\xi_{k,t})-\nabla h_{\eta_k}(\underline u_{k,t})$ and
		\begin{align*}
			Y_{k,t}(u) :=& \lambda_k\beta_k\alpha_{k,t}
			[V(x_{k-1},u)-V(x_{k-1},u_{k,t})]\\
			&+\lambda_k\alpha_{k,t}
			[a_{k,t}V(u_{k,t-1},u)-(\beta_k+a_{k,t}+\mu)V(u_{k,t},u)].
		\end{align*}
	\end{lemma}
	\begin{proof}{Proof}
		By the definition of $Q_k$ in \eqref{def:Q} and the linearity of $g_k(\cdot)$,
		\begin{align}
			& Q_k(\bar{x}_k,u) - (1-\gamma_k)Q_k(\bar{x}_{k-1},u) \nonumber \\
			&= g_k\bigl(\bar{x}_k-\bar{x}_{k-1}+\gamma_k(\bar{x}_{k-1}-u)\bigr) + h_{\eta_k}(\bar{x}_k) - h_{\eta_k}(\bar{x}_{k-1})+ \gamma_k\bigl(h_{\eta_k}(\bar{x}_{k-1})-h_{\eta_k}(u)\bigr). \label{lem3.3:2}
		\end{align}
		Define
		\begin{align}\label{lem3.3:3}
			v_k = (1-\lambda_k)\bar{x}_{k-1} + \lambda_k u, \qquad
			\bar{u}_{k,t} = (1-\lambda_k)\bar{x}_{k-1} + \lambda_k\tilde{u}_{k,t} .
		\end{align}
		Then
		\begin{align}\label{lem3.3:4}
			\gamma_k(\bar{x}_{k-1}-u) = \frac{\gamma_k}{\lambda_k}(\bar{x}_{k-1}-v_k).
		\end{align}
		From the convexity of $h_{\eta_k}$ and \eqref{lem3.3:3} we obtain
		$\frac{\gamma_k}{\lambda_k}\bigl[h_{\eta_k}(v_k) - (1-\lambda_k)h_{\eta_k}(\bar{x}_{k-1}) - \lambda_k h_{\eta_k}(u)\bigr] \le 0$,
		or equivalently,
		\begin{align}\label{lem3.3:4.5}
			\gamma_k\bigl(h_{\eta_k}(\bar{x}_{k-1})-h_{\eta_k}(u)\bigr)
			\le \frac{\gamma_k}{\lambda_k}\bigl(h_{\eta_k}(\bar{x}_{k-1})-h_{\eta_k}(v_k)\bigr).
		\end{align}
		Substituting \eqref{lem3.3:4} and \eqref{lem3.3:4.5} into \eqref{lem3.3:2} and using the definition of $Q_k$ yields
		\begin{align}
			& Q_k\left(\bar{x}_k, u\right)-(1-\gamma_k) Q_k(\bar{x}_{k-1}, u) \nonumber\\
			& \leq g_k\left(\bar{x}_k-\bar{x}_{k-1}+\frac{\gamma_k}{\lambda_k}(\bar{x}_{k-1}-v_k)\right)+h_{\eta_k}\left(\bar{x}_k\right)-h_{\eta_k}(\bar{x}_{k-1})
			+\frac{\gamma_k}{\lambda_k}(h_{\eta_k}(\bar{x}_{k-1})-h_{\eta_k}(v_k))\nonumber \\
			&=Q_k\left(\bar{x}_k, v_k\right)-\left(1-\frac{\gamma_k}{\lambda_k}\right) Q_k(\bar{x}_{k-1}, v_k).\label{lem3.3:6}
		\end{align}
		Since $\tilde{u}_{k,0}=\bar{x}_{k-1}$ and $\tilde{x}_k=\tilde{u}_{k,T_k}$, we have $\bar{x}_k=\bar{u}_{k,T_k}$ and $\bar{u}_{k,0}=\bar{x}_{k-1}$ by \eqref{up2-bxk} and \eqref{lem3.3:3}. Therefore,
		\begin{align}\label{lem3.3:7}
			Q_k(\bar{x}_k,u) - (1-\gamma_k)Q_k(\bar{x}_{k-1},u)
			\le Q_k(\bar{u}_{k,T_k},v_k) - \Bigl(1-\frac{\gamma_k}{\lambda_k}\Bigr)Q_k(\bar{u}_{k,0},v_k).
		\end{align}
		From the definitions of $\underline{u}_{k,t}$, $\tilde{u}_{k,t}$ and $v_k$,
		\begin{align}
			\bar{u}_{k,t} & -\left(1-\alpha_{k,t}\right) \bar{u}_{k,t-1}-\alpha_{k,t} v_k
			=\left(\bar{u}_{k,t}-\bar{u}_{k,t-1}\right)+\alpha_{k,t}\left(\bar{u}_{k,t-1}-v_k\right) \nonumber\\
			& =\lambda_k\left(\tilde{u}_{k,t}-\tilde{u}_{k,t-1}\right)+\lambda_k \alpha_{k,t}\left(\tilde{u}_{k,t-1}-u\right)
			=\lambda_k \alpha_{k,t}\left(u_{k,t}-u\right),\label{lem3.3:8}
		\end{align}
		and
		\begin{align}
			\bar{u}_{k,t}-\underline{u}_{k,t}=\lambda_k\left(\tilde{u}_{k,t}-\left(1-\alpha_{k,t}\right) \tilde{u}_{k,t-1}\right)-\lambda_k \alpha_{k,t} u_{k,t-1}=\lambda_k \alpha_{k,t}\left(u_{k,t}-u_{k,t-1}\right) .\label{lem3.3:9}
		\end{align}
		Using the definition of $Q_k$, the convexity of $h_{\eta_k}$, and the smoothness of $h_{\eta_k}$,
		\begin{align}
			& Q_k\left(\bar{u}_{k,t}, v_k\right)-\left(1-\alpha_{k,t}\right) Q_k\left(\bar{u}_{k,t-1}, v_k\right) \nonumber\\
			& \leq \lambda_k \alpha_{k,t}\left(g_k\left(u_{k,t}\right)-g_k(u)\right)+l_{h_{\eta_k}}\left(\underline{u}_{k,t}, \bar{u}_{k,t}\right)+\frac{L_{\eta_k}}{2}\left\|\bar{u}_{k,t}-\underline{u}_{k,t}\right\|^2 \nonumber\\
			& \quad-\left(1-\alpha_{k,t}\right) l_{h_{\eta_k}}\left(\underline{u}_{k,t}, \bar{u}_{k,t-1}\right)-\alpha_{k,t} l_{h_{\eta_k}}\left(\underline{u}_{k,t}, v_k\right)\nonumber\\
			&=\lambda_k \alpha_{k,t}\left(g_k\left(u_{k,t}\right)-g_k(u)\right)+\left\langle\nabla h_{\eta_k}\left(\underline{u}_{k,t}\right),
			\bar{u}_{k,t}-\left(1-\alpha_{k,t}\right)\bar{u}_{k,t-1}-\alpha_{k,t}v_k\right\rangle\nonumber\\
			&\quad
			+\frac{L_{\eta_k}}{2}\left\|\bar{u}_{k,t}-\underline{u}_{k,t}\right\|^2.\label{lem3.3:10}
		\end{align}
		where $l_{h_{\eta_k}}\left(\underline{u}_{k,t}, u\right)=h_{\eta_k}(\underline{u}_{k,t})+\langle \nabla h_{\eta_k}\left(\underline{u}_{k,t}\right),u-\underline{u}_{k,t}\rangle$. Combining \eqref{lem3.3:8}-\eqref{lem3.3:10} gives
		\begin{align}
			& Q_k\left(\bar{u}_{k,t}, v_k\right)-\left(1-\alpha_{k,t}\right) Q_k\left(\bar{u}_{k,t-1}, v_k\right) \nonumber\\
			& \quad\le\lambda_k \alpha_{k,t}\left[g_k\left(u_{k,t}\right)-g_k(u)
			+l_{h_{\eta_k}}\left(\underline{u}_{k,t}, u_{k,t}\right)
			-l_{h_{\eta_k}}\left(\underline{u}_{k,t}, u\right)\right]+\frac{L_{\eta_k}\lambda_k^2\alpha_{k,t}^2}{2}
			\left\|u_{k,t}-u_{k,t-1}\right\|^2.\label{lem3.3:11}
		\end{align}
		The averaging recursion also yields
		\begin{align}\label{lem3.3:n1}
			\chi(\bar x_k)-(1-\gamma_k)\chi(\bar x_{k-1})-\gamma_k\chi(u)
			\le\lambda_k\Lambda_{k,T_k}\sum_{t=1}^{T_k}\frac{\alpha_{k,t}}{\Lambda_{k,t}}[\chi(u_{k,t})-\chi(u)].
		\end{align}
		We give the proof of \eqref{lem3.3:n1} in detail. By the initialization and output rules of the inner procedure, $\tilde u_{k,0}=\bar x_{k-1}$ and $\tilde x_k=\tilde u_{k,T_k}$. We first claim that, for every $m=1,\ldots,T_k$,
		\begin{align}
			\tilde u_{k,m}
			&=\Lambda_{k,m}(1-\alpha_{k,1})\bar x_{k-1}
			+\Lambda_{k,m}\sum_{t=1}^{m}\frac{\alpha_{k,t}}{\Lambda_{k,t}}u_{k,t}.\label{lem3.3:unroll}
		\end{align}
		For $m=1$, since $\Lambda_{k,1}=1$, the right-hand side of \eqref{lem3.3:unroll} is $(1-\alpha_{k,1})\bar x_{k-1}+\alpha_{k,1}u_{k,1}=\tilde u_{k,1}$. Suppose \eqref{lem3.3:unroll} holds with $m-1$ in place of $m$. Using \eqref{update2-tu} and $\Lambda_{k,m}=(1-\alpha_{k,m})\Lambda_{k,m-1}$ gives
		\begin{align*}
			\tilde u_{k,m}
			&=(1-\alpha_{k,m})\left[\Lambda_{k,m-1}(1-\alpha_{k,1})\bar x_{k-1}
			+\Lambda_{k,m-1}\sum_{t=1}^{m-1}\frac{\alpha_{k,t}}{\Lambda_{k,t}}u_{k,t}\right]
			+\alpha_{k,m}u_{k,m}\\
			&=\Lambda_{k,m}(1-\alpha_{k,1})\bar x_{k-1}
			+\Lambda_{k,m}\sum_{t=1}^{m}\frac{\alpha_{k,t}}{\Lambda_{k,t}}u_{k,t}.
		\end{align*}
		This completes the induction. Taking $m=T_k$ in \eqref{lem3.3:unroll} and substituting $\tilde x_k=\tilde u_{k,T_k}$ into \eqref{up2-bxk}, we obtain
		\begin{align}
			\bar x_k
			&=[1-\lambda_k+\lambda_k\Lambda_{k,T_k}(1-\alpha_{k,1})]\bar x_{k-1}
			+\lambda_k\Lambda_{k,T_k}\sum_{t=1}^{T_k}\frac{\alpha_{k,t}}{\Lambda_{k,t}}u_{k,t}.\label{lem3.3:barxplus-first}
		\end{align}
		The parameter in \eqref{par3:3} implies
		\begin{align*}
			1-\lambda_k+\lambda_k\Lambda_{k,T_k}(1-\alpha_{k,1})
			=1-\lambda_k+\lambda_k\left(1-\frac{\gamma_k}{\lambda_k}\right)
			=1-\gamma_k.
		\end{align*}
		Consequently, \eqref{lem3.3:barxplus-first} becomes
		\begin{align}
			\bar{x}_k
			&=(1-\gamma_k)\bar{x}_{k-1}
			+\gamma_k\sum_{t=1}^{T_k}
			\frac{\lambda_k\Lambda_{k,T_k}\alpha_{k,t}}{\gamma_k\Lambda_{k,t}}u_{k,t}.\label{lem3.3:barxplus}
		\end{align}
		From \eqref{par3:1},
		\begin{align*}
			\frac{\Lambda_{k,T_k}}{\Lambda_{k,t}}
			=\prod_{j=t+1}^{T_k}(1-\alpha_{k,j}),
			\qquad
			\Lambda_{k,T_k}(1-\alpha_{k,1})=\prod_{j=1}^{T_k}(1-\alpha_{k,j}),
		\end{align*}
		Therefore
		\begin{align}
			\Lambda_{k,T_k}\sum_{t=1}^{T_k}\frac{\alpha_{k,t}}{\Lambda_{k,t}}
			&=\sum_{t=1}^{T_k}\alpha_{k,t}\prod_{j=t+1}^{T_k}(1-\alpha_{k,j})=1-\prod_{j=1}^{T_k}(1-\alpha_{k,j})
			=1-\Lambda_{k,T_k}(1-\alpha_{k,1})
			=\frac{\gamma_k}{\lambda_k}.\label{lem3.3:weight-sum}
		\end{align}
		Define
		\[
		s_{k,t}:=\frac{\lambda_k\Lambda_{k,T_k}\alpha_{k,t}}{\gamma_k\Lambda_{k,t}},
		\qquad t=1,\ldots,T_k.
		\]
		Since $0<\alpha_{k,t}<1$ and $0<\gamma_k<\lambda_k$, each $s_{k,t}$ is nonnegative, and \eqref{lem3.3:weight-sum} gives $\sum_{t=1}^{T_k}s_{k,t}=1$. Hence \eqref{lem3.3:barxplus} is a convex combination because $1-\gamma_k\ge0$ and
		$(1-\gamma_k)+\gamma_k\sum_{t=1}^{T_k}s_{k,t}=1$. Convexity of $\chi$ now yields
		\begin{align*}
			\chi(\bar x_k)
			&\le(1-\gamma_k)\chi(\bar x_{k-1})
			+\gamma_k\sum_{t=1}^{T_k}s_{k,t}\chi(u_{k,t})=(1-\gamma_k)\chi(\bar x_{k-1})
			+\lambda_k\Lambda_{k,T_k}\sum_{t=1}^{T_k}
			\frac{\alpha_{k,t}}{\Lambda_{k,t}}\chi(u_{k,t}).
		\end{align*}
		Subtracting $
		(1-\gamma_k)\chi(\bar x_{k-1})+\gamma_k\chi(u)$
		from both sides proves \eqref{lem3.3:n1}.
		Applying Lemma~\ref{lemma:1.1} to the subproblem defining $u_{k,t}$ in \eqref{update2-u} yields
		\begin{align}
			& g_k\left(u_{k,t}\right)-g_k(u)+\tilde{l}_{h_{\eta_k}}\left(\underline{u}_{k,t}, u_{k,t}\right)-\tilde{l}_{h_{\eta_k}}\left(\underline{u}_{k,t}, u\right)+\chi(u_{k,t})-\chi(u)\nonumber\\
			\leq & \beta_k\left(V(x_{k-1}, u)-V\left(u_{k,t}, u\right)-V\left(x_{k-1}, u_{k,t}\right)\right)-\frac{\mu}{2}\|u_{k,t}-u\|^2\nonumber\\
			&+a_{k,t}\left(V\left(u_{k,t-1}, u\right)-V\left(u_{k,t}, u\right)-V\left(u_{k,t-1}, u_{k,t}\right)\right) .\label{lem3.3:12}
		\end{align}
		Denote $\delta_{k,t} = H_{\eta_k}(\underline{u}_{k,t},\xi_{k,t}) - \nabla h_{\eta_k}(\underline{u}_{k,t})$. From the definition of $\tilde{l}_{h_{\eta_k}}$, 
		\begin{align}
			l_{h_{\eta_k}}\left(\underline{u}_{k,t}, u_{k,t}\right)-\tilde{l}_{h_{\eta_k}}\left(\underline{u}_{k,t}, u_{k,t}\right)&=\langle  -\delta_{k,t},u_{k,t}-\underline{u}_{k,t}\rangle,\label{lem3.3:12.2}\\
			\tilde{l}_{h_{\eta_k}}\left(\underline{u}_{k,t}, u\right)-l_{h_{\eta_k}}\left(\underline{u}_{k,t}, u\right)&=\langle  \delta_{k,t},u-\underline{u}_{k,t}\rangle.\label{lem3.3:12.3}
		\end{align}
		Because $V(u_{k,t},u)\le\|u_{k,t}-u\|^2/2$, the last term is at most $-\mu V(u_{k,t},u)$. Moreover, by the strong convexity of $\omega$ and condition \eqref{par3:3},
		\begin{align}\label{lem3.3:13}
			\frac{L_{\eta_k} \lambda_k \alpha_{k,t}}{2}\left\|u_{k,t}-u_{k,t-1}\right\|^2 \leq \frac{L_{\eta_k} \lambda_k \alpha_{k,t}}{\nu} V\left(u_{k,t-1}, u_{k,t}\right).
		\end{align}
		Combining \eqref{lem3.3:11}--\eqref{lem3.3:13} we obtain
		\begin{align}
			&Q_k\left(\bar{u}_{k,t},v_k\right)
			-\left(1-\alpha_{k,t}\right)Q_k\left(\bar{u}_{k,t-1},v_k\right)+\lambda_k\alpha_{k,t}[\chi(u_{k,t})-\chi(u)]\nonumber\\
			\leq& Y_{k,t}(u)+\lambda_k\alpha_{k,t}
			\langle\delta_{k,t},u-u_{k,t}\rangle-\lambda_k\alpha_{k,t}r_{k,t}V(u_{k,t-1},u_{k,t})\nonumber\\
			\le& Y_{k,t}(u)+\lambda_k\alpha_{k,t}\langle\delta_{k,t},u-u_{k,t-1}\rangle+\lambda_k\alpha_{k,t}\|\delta_{k,t}\|_*\|u_{k,t}-u_{k,t-1}\|
			-\frac{\lambda_k\alpha_{k,t}\nu r_{k,t}}{2}\|u_{k,t}-u_{k,t-1}\|^2\nonumber\\
			\le& Y_{k,t}(u)+\lambda_k\alpha_{k,t}\langle\delta_{k,t},u-u_{k,t-1}\rangle
			+\frac{\lambda_k\alpha_{k,t}}{2\nu r_{k,t}}\|\delta_{k,t}\|_*^2,
		\end{align}
		where
		\begin{align*}
			Y_{k,t}(u)={}&\lambda_k\beta_k\alpha_{k,t}
			[V(x_{k-1},u)-V(x_{k-1},u_{k,t})]\\
			&+\lambda_k\alpha_{k,t}
			[a_{k,t}V(u_{k,t-1},u)-(\beta_k+a_{k,t}+\mu)V(u_{k,t},u)].
		\end{align*}
		Applying Lemma~\ref{lem2} then gives
		\begin{align*}
			&Q_k(\bar{x}_k,u)-(1-\gamma_k)Q_k(\bar{x}_{k-1},u)
			+\chi(\bar{x}_k)-(1-\gamma_k)\chi(\bar{x}_{k-1})-\gamma_k\chi(u)\nonumber\\
			&\le
			Q_k(\bar u_{k,T_k},v_k)-\left(1-\frac{\gamma_k}{\lambda_k}\right)Q_k(\bar u_{k,0},v_k)
			+\lambda_k\Lambda_{k,T_k}\sum_{t=1}^{T_k}\frac{\alpha_{k,t}}{\Lambda_{k,t}}
			[\chi(u_{k,t})-\chi(u)]\nonumber\\
			&\le
			\Lambda_{k,T_k}\sum_{t=1}^{T_k}\frac{Y_{k,t}(u)}{\Lambda_{k,t}}
			+\Lambda_{k,T_k}\sum_{t=1}^{T_k}\frac{\lambda_k\alpha_{k,t}}{\Lambda_{k,t}}
			\left[\langle\delta_{k,t},u-u_{k,t-1}\rangle
			+\frac{\|\delta_{k,t}\|_*^2}{2\nu r_{k,t}}\right],
		\end{align*}
		which completes the proof.
		\Halmos
	\end{proof}
	\begin{lemma}\label{lem3.4}
		Suppose Assumption \ref{ass3:f} and Assumption~\ref{ass:chi} hold. Consider the $k$-th call to the inner procedure in Algorithm \ref{alg2}. If condition \eqref{par3:3} holds, and for every $t$
		\begin{align}
			\alpha_{k,t}=\alpha_k,\qquad p_{k,t}=p_k=\frac{1-\alpha_k}{\alpha_k},\qquad q_{k,t}=p_k\mu,\label{par3.4}
		\end{align}
		then we have
		\begin{align}\label{lem3.4:1}
			&Q_k(\bar{x}_k,u)-(1-\gamma_k)Q_k(\bar{x}_{k-1},u)
			+\chi(\bar{x}_k)-(1-\gamma_k)\chi(\bar{x}_{k-1})-\gamma_k\chi(u)\nonumber\\
			&\quad\le [\lambda_k\beta_k+(\lambda_k-\gamma_k)\mu]V(x_{k-1},u)-\lambda_k(\beta_k+\mu) V(x_k,u)-\frac{\nu\beta_k}{2\gamma_k}\|\bar{x}_k-\underline{x}_k\|^2\nonumber\\
			&\qquad+\Lambda_{k,T_k}\sum_{t=1}^{T_k}\frac{\lambda_k\alpha_k}{\Lambda_{k,t}}
			\left[\langle\delta_{k,t},u-u_{k,t-1}\rangle+\frac{\|\delta_{k,t}\|_*^2}{2\nu r_{k,t}}\right],
		\end{align}
	\end{lemma}
	\begin{proof}{Proof}
		Under \eqref{par3.4}, $a_{k,t}=p_k(\beta_k+\mu)$, $\alpha_k a_{k,t}=(1-\alpha_k)(\beta_k+\mu)$, and $\alpha_k(\beta_k+a_{k,t}+\mu)=\beta_k+\mu$. Therefore,
		\begin{align}\label{lem3.4:3}
			&\lambda_k\Lambda_{k,T_k}\sum_{t=1}^{T_k}\frac{\alpha_k}{\Lambda_{k,t}}
			\{a_{k,t}V(u_{k,t-1},u)-(\beta_k+a_{k,t}+\mu)V(u_{k,t},u)\}\nonumber\\
			&\qquad=(\lambda_k-\gamma_k)(\beta_k+\mu)V(x_{k-1},u)-\lambda_k (\beta_k+\mu) V(x_k,u).
		\end{align}
		Also, \eqref{lem3.3:weight-sum} implies
		\begin{align}\label{lem3.4:4}
			\lambda_k\beta_k\Lambda_{k,T_k}\sum_{t=1}^{T_k}\frac{\alpha_k}{\Lambda_{k,t}}V(x_{k-1},u)=\gamma_k\beta_k V(x_{k-1},u).
		\end{align}
		With $s_{k,t}=(\lambda_k\Lambda_{k,T_k}/\gamma_k)(\alpha_k/\Lambda_{k,t})$, strong convexity of the distance-generating function and convexity of $\|\cdot\|^2$ give
		\begin{align}\label{lem3.4:5}
			\lambda_k\beta_k\Lambda_{k,T_k}\sum_{t=1}^{T_k}\frac{\alpha_k}{\Lambda_{k,t}}V(x_{k-1},u_{k,t})
			&\ge\frac{\nu\gamma_k\beta_k}{2}\left\|x_{k-1}-\sum_{t=1}^{T_k}s_{k,t}u_{k,t}\right\|^2
			=\frac{\nu\beta_k}{2\gamma_k}\|\bar x_k-\underline x_k\|^2.
		\end{align}
		Substituting \eqref{lem3.4:3} --\eqref{lem3.4:5} into Lemma~\ref{lem3.3} gives \eqref{lem3.4:1}, because
		$\gamma_k\beta_k+(\lambda_k-\gamma_k)(\beta_k+\mu)=\lambda_k\beta_k+(\lambda_k-\gamma_k)\mu$. 		
		\Halmos
	\end{proof}
	With the help of Lemma~\ref{lem3.3} and Lemma~\ref{lem3.4}, we are now ready to establish the convergence of Algorithm~\ref{alg2}.
	\begin{theorem}\label{thm3.5}
		Suppose Assumptions \ref{ass:chi}, \ref{ass:smoothing}, \ref{ass3:f}, and \ref{ass:h2} hold. Define, for every outer iteration $k$,
		\begin{align}\label{par3.51}
			&\widehat L=\max\{L,4\nu\mu\},\quad
			\lambda_k=\lambda:=\sqrt{\frac{\nu\mu}{\widehat L}},\quad\gamma_k=\gamma:=\frac{\lambda}{2},\quad
			\beta_k=\beta:=\frac{\widehat L\gamma}{\nu}.
		\end{align}
		Let $x^*$ be a solution of the original problem \eqref{prob} and let $\Delta_0>0$ satisfy
		\begin{align}\label{par3:Delta0}
			\Delta_0\ge\Psi(\bar x_0)-\Psi^*+\lambda(\beta+\mu)V(x_0,x^*).
		\end{align}
		Let $0<\epsilon<8\Delta_0$.
		Set
		\begin{align}\label{par3.5}
			&N=\left\lceil\frac{2\log(8\Delta_0/\epsilon)}{-\log(1-\gamma)}\right\rceil,
			\quad \eta_k=\frac{\Delta_0}{B_h}(1-\gamma)^{k/2},\quad
			L_{\eta_k}=\frac{C_h}{\eta_k},\nonumber\\
			&T_k^{\rm g}:=\left\lceil\frac{\log(1/2)}{\log\bigl(1-\frac{1}4\sqrt{\frac{\widehat L}{\max\{\widehat L,L_{\eta_k}\}}}\bigr)}\right\rceil,\quad T_k^{\rm s}
			=
			\left\lceil
			\frac{8\sigma^2\log 2(1-(1-\gamma)^{N/2})}{5\widehat L\epsilon(1-\sqrt{1-\gamma})}
			(1-\gamma)^{(N-k)/2}
			\right\rceil, \nonumber\\
			&T_k:=\max\{T_k^{\rm g},T_k^{\rm s}\},\quad\alpha_{k,t}=\alpha_k=1-2^{-1/T_k},\quad p_{k,t}=p_k=\frac{1-\alpha_k}{\alpha_k},\quad q_{k,t}=p_k\mu.
		\end{align}
		Then we have
		\begin{align}\label{thm3.5:1}
			\mathbb{E}\bigl[\Psi(\bar{x}_N) - \Psi(x^*)\bigr]
			\le \epsilon.
		\end{align}
		Let $N_f$ denote the number of evaluations of $\nabla f$, and let $N_H$ denote the number of stochastic-oracle evaluations. Then
		\begin{align}\label{thm3.5:complexities}
			N_f&=\mathcal O\!\left(
			\left[1+\sqrt{\frac{L}{\nu\mu}}\right]
			\left[1+\log\max\left\{\frac{\Delta_0}{\epsilon},1\right\}\right]\right),\quad
			N_H
			=\mathcal O\!\left(
			N_f+
			\sqrt{\frac{B_hC_h}{\nu\mu\epsilon}}
			+\frac{\sigma^2}{\nu\mu\epsilon}\right).
		\end{align}
		In the exact-oracle case $\sigma=0$, the number of evaluations of the smoothed gradients is
		\begin{align}\label{thm3.5:Kcomplexity}
			N_{\nabla h_\eta}
			=\mathcal O\!\left(
			N_f+\sqrt{\frac{B_hC_h}{\nu\mu\epsilon}}
			\right).
		\end{align}
	\end{theorem}
	\begin{proof}{Proof}
		We first verify \eqref{par3:3} in Lemma~\ref{lem3.3}. From \eqref{par3.51},
		\begin{align}
			0<\lambda_k=\lambda=\sqrt{\frac{\nu\mu}{\widehat L}}\le\frac12,\quad
			\gamma_k=\gamma=\frac{\lambda}{2},\quad
			\beta_k=\beta=\frac{\widehat L\lambda}{2\nu},\quad
			\lambda_k\beta_k=\lambda\beta=\frac{\mu}{2}.\label{thm3.5:parameter-basic}
		\end{align}
		Since \(T_k\ge T_k^{\rm g}\), the definition of \(T_k^{\rm g}\) in \eqref{par3.5} gives
		\begin{align}
			T_k
			&\ge
			\frac{\log(1/2)}
			{\log\!\left(1-\frac14\sqrt{\frac{\widehat L}{\max\{\widehat L,L_{\eta_k}\}}}\right)}\Longrightarrow\quad
			\alpha_k=1-2^{-1/T_k}
			\le\frac14\sqrt{\frac{\widehat L}{\max\{\widehat L,L_{\eta_k}\}}}\le\frac14.
			\label{thm3.5:2}
		\end{align}
		Moreover, \eqref{par3:1}, \eqref{par3.5}, and \(\gamma_k/\lambda_k=1/2\) imply
		\begin{align}
			\Lambda_{k,T_k}(1-\alpha_{k,1})
			&=(1-\alpha_k)^{T_k}
			=\left(2^{-1/T_k}\right)^{T_k}
			=\frac12
			=1-\frac{\gamma_k}{\lambda_k}.
			\label{thm3.5:weight-condition}
		\end{align}
		Since \(a_{k,t}=\beta_k p_{k,t}+q_{k,t}=p_k(\beta_k+\mu)\), by \eqref{thm3.5:2}, we have 
		\begin{align}
			p_k=\frac{1-\alpha_k}{\alpha_k}
			&\ge\frac{3}{4\alpha_k},\quad
			a_{k,t}\ge p_k\beta_k
			\ge\frac{3\widehat L\lambda}{8\nu\alpha_k},\quad
			\frac{\lambda_k L_{\eta_k}\alpha_{k,t}}{\nu}
			\le\frac{\lambda \widehat L\alpha_k^2}{\nu\alpha_k}\le
			\frac{\widehat L\lambda}{16\nu\alpha_k}.
			\label{thm3.5:parameter-slack}
		\end{align}
		Consequently,
		\begin{align}
			r_{k,t}:=a_{k,t}-\frac{\lambda_k L_{\eta_k}\alpha_{k,t}}{\nu}
			&\ge
			\frac{5\widehat L\lambda}{16\nu\alpha_k}>0.
			\label{thm3.5:3}
		\end{align}
		\eqref{thm3.5:parameter-basic}, \eqref{thm3.5:weight-condition}, and \eqref{thm3.5:3} verify \eqref{par3:3}. 
		Substituting \eqref{lem3.4:1} into \eqref{lem3.2:1} with \(u=x^*\) yields
		\begin{align}
			\phi_k(\bar x_k)-\phi_k(x^*)
			&\le
			(1-\gamma_k)\,[\phi_k(\bar x_{k-1})-\phi_k(x^*)]
			+[\lambda_k\beta_k+(\lambda_k-\gamma_k)\mu]V(x_{k-1},x^*)\nonumber\\
			&\quad-\lambda_k(\beta_k+\mu)V(x_k,x^*)+
			\left(\frac L2-\frac{\nu\beta_k}{2\gamma_k}\right)
			\|\bar x_k-\underline x_k\|^2+\zeta_k,
			\label{thm3.5:one-step-raw}
		\end{align}
		where $\zeta_k
		=\Lambda_{k,T_k}\sum_{t=1}^{T_k}
		\frac{\lambda_k\alpha_{k,t}}{\Lambda_{k,t}}
		\left[
		\langle\delta_{k,t},x^*-u_{k,t-1}\rangle
		+\frac{\|\delta_{k,t}\|_*^2}{2\nu r_{k,t}}
		\right]$. The two coefficients in \eqref{thm3.5:one-step-raw} satisfy
		\begin{align}
			\frac L2-\frac{\nu\beta_k}{2\gamma_k}
			&=\frac{L-\widehat L}{2}\le0,\nonumber\\
			(1-\gamma_k)\lambda_k(\beta_k+\mu)
			&-\lambda_k\beta_k-(\lambda_k-\gamma_k)\mu
			=\gamma_k\left(\frac12-\lambda_k\right)\mu\ge0.
			\label{thm3.5:coefficient-two}
		\end{align}
		Define the Lyapunov quantity
		\begin{align}
			\mathcal E_k
			:=\phi_k(\bar x_k)-\phi_k(x^*)+\lambda(\beta+\mu)V(x_k,x^*).
			\label{thm3.5:energy}
		\end{align}
		Adding \(\lambda(\beta+\mu)V(x_k,x^*)\) to \eqref{thm3.5:one-step-raw} and using \eqref{thm3.5:coefficient-two} gives
		\begin{align}
			\mathcal E_k
			&\le
			(1-\gamma)\,[\phi_k(\bar x_{k-1})-\phi_k(x^*)
			+\lambda(\beta+\mu)V(x_{k-1},x^*)]+\zeta_k.
			\label{thm3.5:4}
		\end{align}
		Since \(\eta_k\le\eta_{k-1}\), \eqref{rem:smoothing-cross} gives
		$0\le h_{\eta_k}(x)-h_{\eta_{k-1}}(x)\le B_h(\eta_{k-1}-\eta_k)$ for every $x\in X$. Hence
		\begin{align}
			\phi_k(\bar x_{k-1})-\phi_k(x^*)
			&\le
			\phi_{k-1}(\bar x_{k-1})-\phi_{k-1}(x^*)
			+B_h(\eta_{k-1}-\eta_k).
			\label{thm3.5:5}
		\end{align}
		Combining \eqref{thm3.5:4} and \eqref{thm3.5:5},
		\begin{align}
			\mathcal E_k
			&\le
			(1-\gamma)\mathcal E_{k-1}
			+(1-\gamma)B_h(\eta_{k-1}-\eta_k)+\zeta_k.
			\label{thm3.5:pathwise-recursion}
		\end{align}
		By Assumption~\ref{ass:h2}, we have
		\begin{align}
			\mathbb E\!\left[
			\langle\delta_{k,t},x^*-u_{k,t-1}\rangle
			\mid\mathcal F_{k,t-1}\right]=0,\quad
			\mathbb E[\|\delta_{k,t}\|_*^2\mid\mathcal F_{k,t-1}]\le\sigma^2.
			\label{thm3.5:variance}
		\end{align}
		Define $
		b_{k,t}
		:=\frac{\lambda\alpha_k\Lambda_{k,T_k}}
		{2\nu r_{k,t}\Lambda_{k,t}}$.
		Taking expectations in \eqref{thm3.5:pathwise-recursion} and using \eqref{thm3.5:variance}, we obtain
		\begin{align}
			\mathbb E[\mathcal E_k]
			&\le
			(1-\gamma)\,\mathbb E[\mathcal E_{k-1}]
			+(1-\gamma)B_h(\eta_{k-1}-\eta_k)
			+\sigma^2\sum_{t=1}^{T_k}b_{k,t}.
			\label{thm3.5:6}
		\end{align}
		By \eqref{thm3.5:3}, we get $
		b_{k,t}
		\le
		\frac{8\alpha_k^2}{5\widehat L}
		\frac{\Lambda_{k,T_k}}{\Lambda_{k,t}}$.
		Since \(\Lambda_{k,T_k}/\Lambda_{k,t}
		=(1-\alpha_k)^{T_k-t}\) and
		\((1-\alpha_k)^{T_k}=1/2\), then we have
		\begin{align}
			\sum_{t=1}^{T_k}b_{k,t}
			&\le
			\frac{8\alpha_k^2}{5\widehat L}
			\sum_{j=0}^{T_k-1}(1-\alpha_k)^j=
			\frac{8\alpha_k^2}{5\widehat L}
			\frac{1-(1-\alpha_k)^{T_k}}{\alpha_k}
			=\frac{4\alpha_k}{5\widehat L}\nonumber\\
			&=
			\frac{4}{5\widehat L}
			\left(1-e^{-\log(2)/T_k}\right)
			\le\frac{4\log(2)}{5\widehat L T_k}.
			\label{thm3.5:7}
		\end{align}
		Summing up \eqref{thm3.5:6} from \(k=1\) to \(N\), and then applying \eqref{thm3.5:7}, gives
		\begin{align}
			\mathbb E[\mathcal E_N]
			&\le
			(1-\gamma)^N\mathcal E_0
			+B_h\sum_{k=1}^N
			(1-\gamma)^{N-k+1}(\eta_{k-1}-\eta_k) +
			\frac{4\sigma^2\log(2)}{5\widehat L}
			\sum_{k=1}^N\frac{(1-\gamma)^{N-k}}{T_k}.
			\label{thm3.5:unrolled}
		\end{align}
		For \(\sigma>0\), \(T_k\ge T_k^{\rm s}\) in \eqref{par3.5} implies
		$
		\frac{1}{T_k}
		\le
		\frac{5\widehat L\epsilon(1-\sqrt{1-\gamma})}
		{8\sigma^2\log(2)(1-(1-\gamma)^{N/2})}
		(1-\gamma)^{-(N-k)/2}$.
		Therefore,
		\begin{align}
			\frac{4\sigma^2\log(2)}{5\widehat L}
			\sum_{k=1}^N\frac{(1-\gamma)^{N-k}}{T_k}
			&\le
			\frac{\epsilon(1-\sqrt{1-\gamma})}{2(1-(1-\gamma)^{N/2})}
			\sum_{k=1}^N(1-\gamma)^{(N-k)/2}
			=\frac{\epsilon}{2}.
			\label{thm3.5:noise-sum}
		\end{align}
		When \(\sigma=0\), the left-hand side of \eqref{thm3.5:noise-sum} is zero, so the same bound holds.
		Next, \(\eta_k=\frac{\Delta_0}{B_h}(1-\gamma)^{k/2}\) gives
		\begin{align}
			\sum_{k=1}^N
			(1-\gamma)^{N-k+1}(\eta_{k-1}-\eta_k)
			&=
			\sqrt{(1-\gamma)}\,\eta_N(1-(1-\gamma)^{N/2})
			\le\sqrt{(1-\gamma)}\,\eta_N.
			\label{thm3.5:smoothing-sum}
		\end{align}
		Substituting \eqref{thm3.5:noise-sum} and \eqref{thm3.5:smoothing-sum} into \eqref{thm3.5:unrolled},
		\begin{align}
			\mathbb E[\mathcal E_N]
			&\le
			(1-\gamma)^N\mathcal E_0+\sqrt{(1-\gamma)}\,\eta_NB_h
			+\frac{\epsilon}{2}.
			\label{thm3.5:8}
		\end{align}
		Since
		\(\phi_k(x)\le\Psi(x)\le\phi_k(x)+B_h\eta_k\) by \eqref{smooth:transfer}, then we have
		\begin{align}
			\mathcal E_0
			&=
			\phi_0(\bar x_0)-\phi_0(x^*)+\lambda(\beta+\mu)V(x_0,x^*)\nonumber\\
			&\le
			\Psi(\bar x_0)-\Psi^*+\lambda(\beta+\mu)V(x_0,x^*)+\eta_0B_h
			\le2\Delta_0,
			\label{thm3.5:initial-energy}\\
			\Psi(\bar x_N)-\Psi^*
			&\le
			\phi_N(\bar x_N)-\phi_N(x^*)+\eta_NB_h
			\le\mathcal E_N+\eta_NB_h.
			\label{thm3.5:terminal-transfer}
		\end{align}
		For \(N\ge1\), the definition of \(N\) in \eqref{par3.5} yields
		\begin{align}
			(1-\gamma)^{N/2}
			&\le\frac{\epsilon}{8\Delta_0},\qquad
			\eta_NB_h=\Delta_0(1-\gamma)^{N/2}.
			\label{thm3.5:N-choice}
		\end{align}
		\eqref{thm3.5:N-choice} gives
		\begin{align}
			2\Delta_0(1-\gamma)^N
			&\le\frac{\epsilon}{4}(1-\gamma)^{N/2}\le\frac{\epsilon}{4},\nonumber\\
			(1+\sqrt{1-\gamma})\eta_NB_h
			&=(1+\sqrt{1-\gamma})\Delta_0(1-\gamma)^{N/2}
			\le2\Delta_0(1-\gamma)^{N/2}\le\frac{\epsilon}{4}.
			\label{thm3.5:deterministic-sum}
		\end{align}
		Finally, \eqref{thm3.5:8}--\eqref{thm3.5:deterministic-sum} give
		\begin{align}
			\mathbb E[\Psi(\bar x_N)-\Psi^*]
			&\le
			(1-\gamma)^N\mathcal E_0+(1+\sqrt{1-\gamma})\eta_NB_h
			+\frac{\epsilon}{2}\le\epsilon,
		\end{align}
		which proves \eqref{thm3.5:1}. We now count the oracle calls. Since
		\(-\log(1-\gamma)\ge\gamma\), \eqref{par3.5} gives
		\begin{align}
			N_f=N
			&\le
			1+\frac{2}{\gamma}
			\log\left(\frac{8\Delta_0}{\epsilon}\right)=
			\mathcal O\!\left(
			\left[1+\sqrt{\frac{L}{\nu\mu}}\right]
			\left[1+\log\max\left\{\frac{\Delta_0}{\epsilon},1\right\}\right]\right),
			\label{thm3.5:Nf-bound}
		\end{align}
		where
		\(1/\gamma=2\sqrt{\widehat L/(\nu\mu)}
		=\mathcal O(1+\sqrt{L/(\nu\mu)})\). The definition of \(T_k^{\rm g}\) gives
		\begin{align}
			T_k^{\rm g}
			&\le
			1+4\log(2)\sqrt{\frac{\max\{\widehat L,L_{\eta_k}\}}{\widehat L}}=
			\mathcal O\left(
			1+\sqrt{\frac{C_h}{\widehat L\eta_k}}
			\right).
			\label{thm3.5:9}
		\end{align}
		For \(N\ge1\), moreover,
		\begin{align}
			\sum_{k=1}^N\eta_k^{-1/2}
			&=
			\eta_N^{-1/2}\sum_{k=1}^N(1-\gamma)^{(N-k)/4}
			\le\frac{\eta_N^{-1/2}}{1-(1-\gamma)^{1/4}}
			\le\frac{4}{\gamma\sqrt{\eta_N}},
			\label{thm3.5:eta-sum}
		\end{align}
		where \(1-(1-\gamma)^{1/4}\ge\gamma/4\). The minimal integer choice of \(N\) gives
		\begin{align}
			\frac{\sqrt{1-\gamma}\,\epsilon}{8B_h}
			<\eta_N=\frac{\Delta_0}{B_h}(1-\gamma)^{N/2}
			\le\frac{\epsilon}{8B_h}.
			\label{thm3.5:eta-terminal}
		\end{align}
		Thus \eqref{thm3.5:9}--\eqref{thm3.5:eta-terminal}, together with
		\(\gamma\sqrt{\widehat L}=\sqrt{\nu\mu}/2\), yield
		\begin{align}
			\sum_{k=1}^NT_k^{\rm g}
			&=
			\mathcal O\!\left(
			N_f+\sqrt{\frac{B_hC_h}{\nu\mu\epsilon}}
			\right).
			\label{thm3.5:Tg-sum}
		\end{align}
		For $T_k^{\rm s}$, we have
		\begin{align}
			\sum_{k=1}^NT_k^{\rm s}
			&\le
			N+\frac{8\sigma^2\log(2)}
			{5\widehat L\epsilon}\mathcal S_N
			\sum_{k=1}^N(1-\gamma)^{(N-k)/2}=
			N+\frac{8\sigma^2\log(2)}
			{5\widehat L\epsilon}\mathcal S_N^2,
			\label{thm3.5:Ts-pre}
		\end{align}
		where $\mathcal S_N=
		\sum_{j=1}^N(1-\gamma)^{(N-j)/2}
		\le\frac{1}{1-\sqrt{1-\gamma}}\le\frac{2}{\gamma}$. Combining $\widehat L\gamma^2
		=\frac{\nu\mu}{4}$, we get
		\begin{align}
			\sum_{k=1}^NT_k^{\rm s}
			&=
			\mathcal O\!\left(
			N_f+\frac{\sigma^2}{\nu\mu\epsilon}
			\right).
			\label{thm3.5:Ts-sum}
		\end{align}
		Finally, \(T_k=\max\{T_k^{\rm g},T_k^{\rm s}\}
		\le T_k^{\rm g}+T_k^{\rm s}\). Hence \eqref{thm3.5:Nf-bound},
		\eqref{thm3.5:Tg-sum}, and \eqref{thm3.5:Ts-sum} give
		\begin{align}\label{thm3.5:Tg-0}
			N_H
			=\sum_{k=1}^NT_k
			&=
			\mathcal O\!\left(
			N_f+\sqrt{\frac{B_hC_h}{\nu\mu\epsilon}}
			+\frac{\sigma^2}{\nu\mu\epsilon}
			\right),
		\end{align}
		which proves \eqref{thm3.5:complexities}. If
		\(\sigma=0\), then \eqref{thm3.5:Tg-0} gives
		\begin{align}
			N_{\nabla h_\eta}
			=\sum_{k=1}^NT_k^{\rm g}
			=\mathcal O\!\left(
			N_f+\sqrt{\frac{B_hC_h}{\nu\mu\epsilon}}
			\right),
		\end{align}
		which is \eqref{thm3.5:Kcomplexity}.		
		\Halmos
	\end{proof}
	\begin{remark}
		The bound $\Delta_0$ in \eqref{par3:Delta0} is computable whenever a lower bound $\Psi_{\rm lb}\le\Psi^*$ and a Bregman-radius bound $D_X\ge\sup_{x\in X}V(x_0,x)<\infty$ are available. Indeed, the choice
		\begin{align*}
			\Delta_0:=\Psi(\bar x_0)-\Psi_{\rm lb}+\lambda(\beta+\mu)D_X
		\end{align*}
		satisfies \eqref{par3:Delta0}. 
	\end{remark}
	\begin{remark}
		Under the assumptions and notation of Theorem~\ref{thm3.5}, let $0<\epsilon<8\Delta_0$ and apply Algorithm~\ref{alg2} with the parameter choices in \eqref{par3.51}--\eqref{par3.5}. If $\nu$, $B_h$, $C_h$, and $\Delta_0$ are treated as fixed problem constants, the numbers of gradient evaluations of $\nabla f$ and stochastic-oracle evaluations required to obtain an $\epsilon$-solution of problem \eqref{prob} are bounded, respectively, by
		\begin{align*}
			&	\mathcal{O}\!\left(\sqrt{\frac{L}{\mu}}\,
			\log\max\!\Bigl\{1,\frac{1}{\epsilon}\Bigr\}\right),\quad \mathcal O\!\left(
			\sqrt{\frac{L}{\mu}}\,
			\log\max\!\Bigl\{1,\frac{1}{\epsilon}\Bigr\}+\frac{1}{\sqrt{\mu\epsilon}}+\frac{\sigma^2}{\mu\epsilon}\right).
		\end{align*}
		If \(\sigma=0\), the above bounds are 
		\begin{align*}
			&	\mathcal{O}\!\left(\sqrt{\frac{L}{\mu}}\,
			\log\max\!\Bigl\{1,\frac{1}{\epsilon}\Bigr\}\right),\quad \mathcal O\!\left(
			\sqrt{\frac{L}{\mu}}\,
			\log\max\!\Bigl\{1,\frac{1}{\epsilon}\Bigr\}+\frac{1}{\sqrt{\mu\epsilon}}\right)
		\end{align*}
		for evaluations of $\nabla f$ and $\nabla h_{\eta_k}$, respectively. Combining Theorem~\ref{thm3.5} with the equivalent reformulation in Remark~\ref{remark:2} shows that RF-ASSGS also covers the setting treated by M-AGS in which $f$ is strongly convex.
	\end{remark}

	\begin{remark}\label{remark:smoothing-examples}
		The following examples satisfy Assumption~\ref{ass:smoothing}.
		
		\begin{itemize}
			\item Consider the max-form function studied in \cite{Lan2022}:
			\begin{align}\label{pro:spp}
				h(x)=\max_{y\in Y}\{\langle Kx,y\rangle-J(y)\},
			\end{align}
			where $Y\subseteq\mathbb R^m$ is closed and convex, $K:\mathbb R^n\to\mathbb R^m$ is linear, and $J:Y\to\mathbb R$ is convex. Let $s:Y\to\mathbb R$ be $\mu_s$-strongly convex, and define
			\begin{align*}
				d(y):=s(y)-s(y_s)-\langle\nabla s(y_s),y-y_s\rangle,
				\qquad
				\Omega:=\max_{y\in Y}d(y)<\infty.
			\end{align*}
			Nesterov's smoothing
			\begin{align}\label{remark:max-smoothing}
				h_\eta(x):=\max_{y\in Y}
				\{\langle Kx,y\rangle-J(y)-\eta d(y)\}
			\end{align}
			has a unique maximizer $y_\eta(x)$ and satisfies \eqref{rem:smoothing-family}--\eqref{rem:smoothing-cross} with
			\begin{align*}
				B_h=\Omega,\qquad
				C_h=\frac{\|K\|^2}{\mu_s},\qquad
				\nabla h_\eta(x)=K^\top y_\eta(x).
			\end{align*}
			Therefore, Theorem~\ref{thm3.5} gives
			\begin{align*}
				N_H=\mathcal O\!\left(
				N_f+\frac{\sqrt{\Omega}\,\|K\|}
				{\sqrt{\nu\mu\mu_s\epsilon}}
				+\frac{\sigma^2}{\nu\mu\epsilon}\right).
			\end{align*}
			If $\nabla h_\eta$ is evaluated exactly using one application each of $K$ and $K^\top$, then $\sigma=0$ and
			\begin{align*}
				N_K=\mathcal O\!\left(
				N_f+\frac{\sqrt{\Omega}\,\|K\|}
				{\sqrt{\nu\mu\mu_s\epsilon}}\right).
			\end{align*}
			Compared with the multi-stage AGS method in \cite{Lan2022}, RF-ASSGS for sloving \eqref{remark:max-smoothing} has three distinguishing features: it uses a continuously updated parameter schedule instead of periodic restarts; it accommodates stochastic, rather than only exact, first-order information for $h_\eta$; and its analysis requires only the abstract smoothing conditions \eqref{rem:smoothing-family}--\eqref{rem:smoothing-cross}, rather than the max-form representation. Moreover, Remark~\ref{remark:2} shows that the strong convexity assumption on $\chi$ also covers the setting in which $f$ is strongly convex. Thus, RF-ASSGS applies to both strong-convexity configurations while retaining a restart-free structure.
			
			\item For
			\[
			h(x)=\max_{1\le i\le m_h}\{a_i^{\!\top}x+b_i\},
			\]
			the shifted log-sum-exp smoothing \cite{Nesterov05}
			\[
			h_\eta(x)=\eta\log\!\left(
			\frac{1}{m_h}\sum_{i=1}^{m_h}
			\exp\!\left(\frac{a_i^{\!\top}x+b_i}{\eta}\right)
			\right)
			\]
			satisfies \eqref{rem:smoothing-family}--\eqref{rem:smoothing-cross} with $B_h=\log m_h$ and $C_h=\|A_h\|_2^2$ under the Euclidean norm, where the $i$th row of $A_h$ is $a_i^{\!\top}$.
			
			\item If $h:\mathbb R^n\to\mathbb R$ is convex and $M$-Lipschitz continuous with respect to the Euclidean norm, its Moreau envelope \cite{Bauschke}
			\[
			h_\eta(x)=\min_{z\in\mathbb R^n}
			\left\{h(z)+\frac{1}{2\eta}\|z-x\|_2^2\right\}
			\]
			satisfies \eqref{rem:smoothing-family}--\eqref{rem:smoothing-cross} with $B_h=M^2/2$ and $C_h=1$. Moreover,
$\nabla h_\eta(x)=\frac{x-\operatorname{prox}_{\eta h}(x)}{\eta},$
			so this construction applies whenever $\operatorname{prox}_{\eta h}$ can be computed efficiently.
		\end{itemize}
	\end{remark}

	\section{Numerical Experiments}\label{sec4}
	In this section, we conduct numerical experiments to compare the performance of the proposed restart-free algorithms, RF-SGS and RF-ASSGS, with the restart-based algorithms, the multi-phase stochastic gradient sliding (M-SGS) and the multi-phase accelerated gradient sliding (M-AGS) algorithms. The comparison is carried out on two classes of problems: portfolio optimization and generalized fused-lasso regression. All experiments are implemented in Python 3.9 and executed on a laptop equipped with an Apple M4 Pro processor and 24 GB of RAM.
	
	\subsection{Portfolio Optimization Problem}
	
	We first evaluate RF-SGS on the portfolio optimization problem \cite{Kremer}
	\begin{align}
		\min_{w\in X}\;
		\frac{\tau}{2}w^{\!\top}\Sigma w-q^{\!\top}w+J_{\boldsymbol\lambda}(w),
	\end{align}
	where
	\[
	X=\{w\in\mathbb R^n:\boldsymbol 1^{\!\top}w=1,\ w_i\ge0\},
	\qquad
	J_{\boldsymbol\lambda}(w)
	=\sum_{i=1}^n\lambda_i^{\rm S}|w|_{(i)},
	\]
	$|w|_{(1)}\ge\cdots\ge|w|_{(n)}$, and
	$\lambda_i^{\rm S}=0.005\cdot10^{-(i-1)/(n-1)}$. Here, $q$ is the vector of expected returns and $\tau>0$ is a regularization parameter. We use the modified covariance matrix
	$\Sigma=\Sigma_0+\mu I_n$ with $\mu>0$,
	which is positive definite and promotes investment diversity \cite{Davis}. We set $\tau=1$ and construct $\Sigma_0$ so that its condition number is approximately $8000$. With $\mu=0.1$, the condition number of $\Sigma$ decreases to approximately $5$. The eigenvectors of $\Sigma_0$ are generated from the QR factorization of a Gaussian matrix, its eigenvalues are geometrically distributed to obtain the prescribed condition numbers, and $q_i\sim\mathcal N(0.03,0.02^2)$. The initial point is $w_0=\boldsymbol 1/n$.
	
	We compare RF-SGS with the multi-phase stochastic gradient sliding method (M-SGS) of \cite{Lan_sliding}. We use the Euclidean distance-generating function
	\begin{align*}
		\omega(w)&:=\frac12\|w\|_2^2,
		\qquad
		V(a,b):=\frac12\|b-a\|_2^2.
	\end{align*}
	For RF-SGS, we set
	\begin{align*}
		f(w)&:=\frac{\tau}{2}w^{\!\top}\Sigma_0w-q^{\!\top}w,
		\qquad
		h(w):=J_{\boldsymbol\lambda}(w),
		\qquad
		\chi(w):=0.05\|w\|_2^2.
	\end{align*}
	The parameters of RF-SGS are selected according to its theoretical prescription: $\beta=L(1-c)$, $\gamma=1-c$, $T_k =\left\lceil \frac{1}{c^{k/2}}\cdot \frac{(\beta+\mu)(1-c)}{ c(\beta+\mu) -\beta}\right\rceil$, $p_{k,t} = \frac{\beta+\mu}{\beta} \cdot  \frac{1}{c^{k/2}}$, $\theta_{k,t} = \frac{1- \frac{1}{1+c^{k/2}} }{  1- \frac{1}{(1+c^{k/2})^t} }$ with $c=\frac{\sqrt{\frac{L}{\mu}}}{1+\sqrt{\frac{L}{\mu}}}$. 
	The parameters of M-SGS are chosen as in \cite{Lan_sliding}. For each dimension, we estimate $\Psi^*$ using an independent exact-oracle RF-SGS run terminated after $N_f^{\rm ref}=1000$ evaluations of $\nabla f$.
	
	\begin{figure}[t]
		\centering
		\begin{tabular}{@{}cc@{}}
			\subfigure[Inner iterations; $n=1000$]{
				\label{fig4.1a}
				\includegraphics[width=0.43\textwidth]{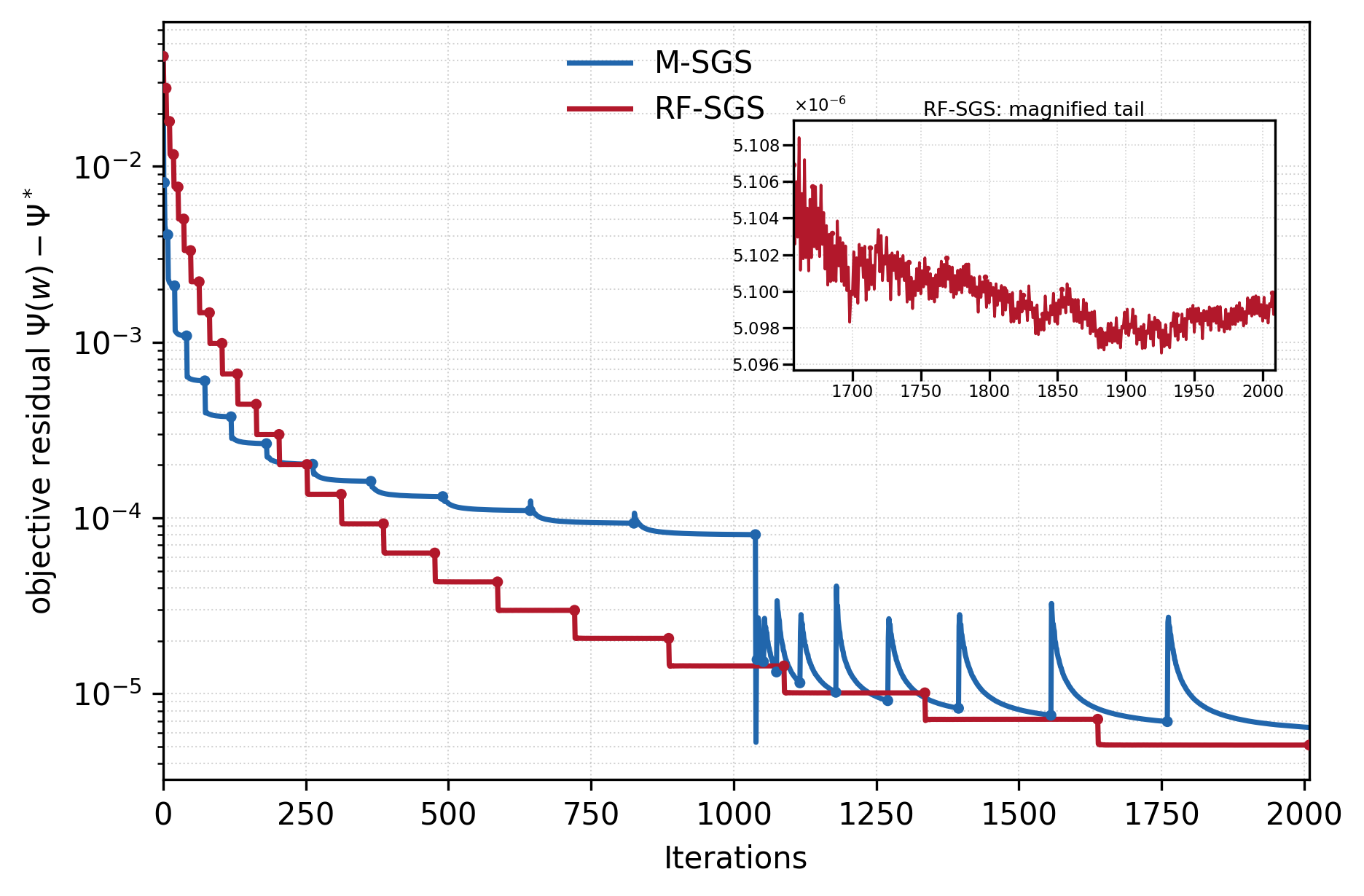}}
			&\subfigure[Inner iterations; $n=2000$]{
				\label{fig4.1b}
				\includegraphics[width=0.43\textwidth]{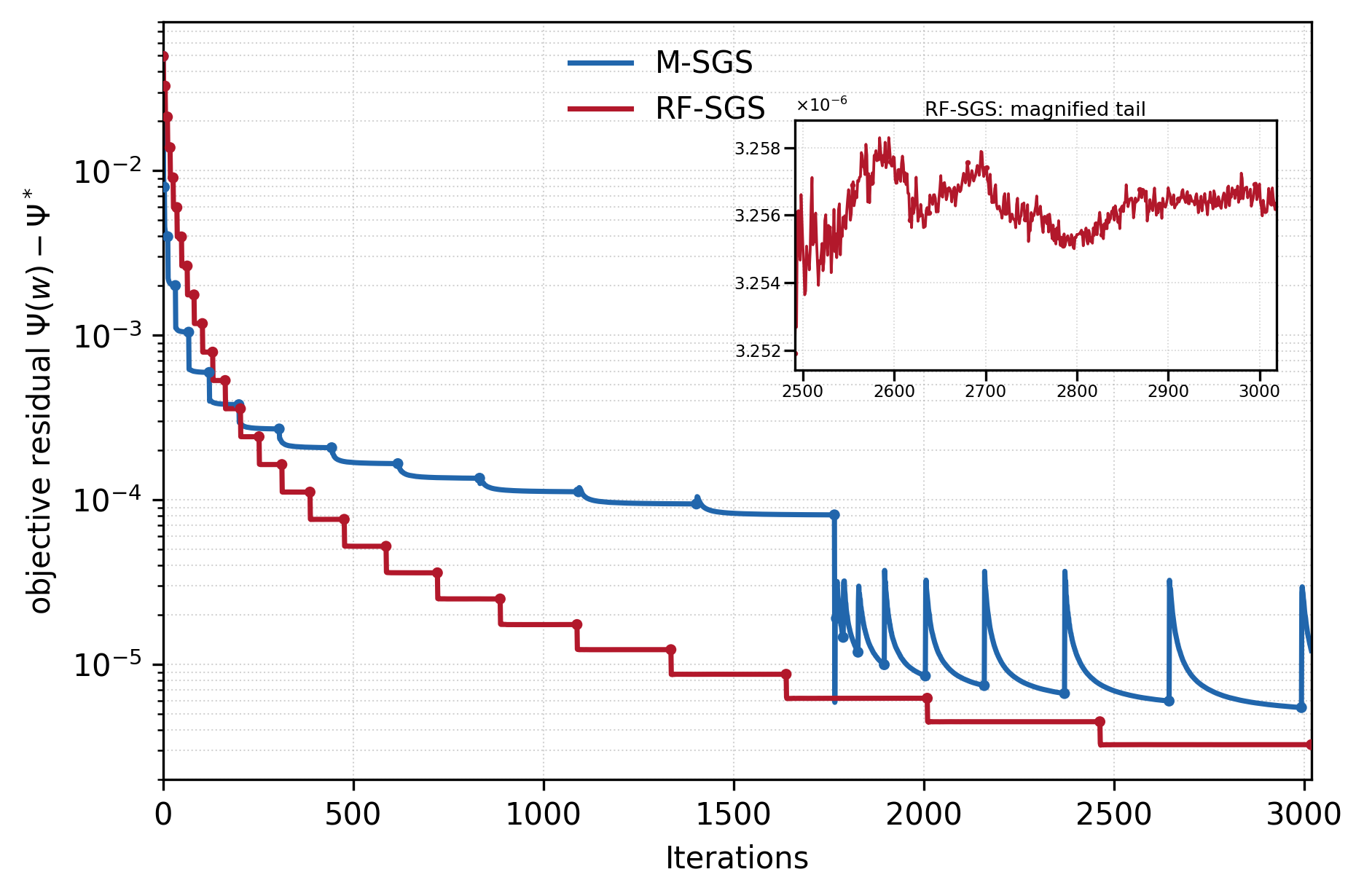}}\\[1ex]
			\subfigure[CPU time; $n=1000$]{
				\label{fig:portfolio-cpu-a}
				\includegraphics[width=0.43\textwidth]{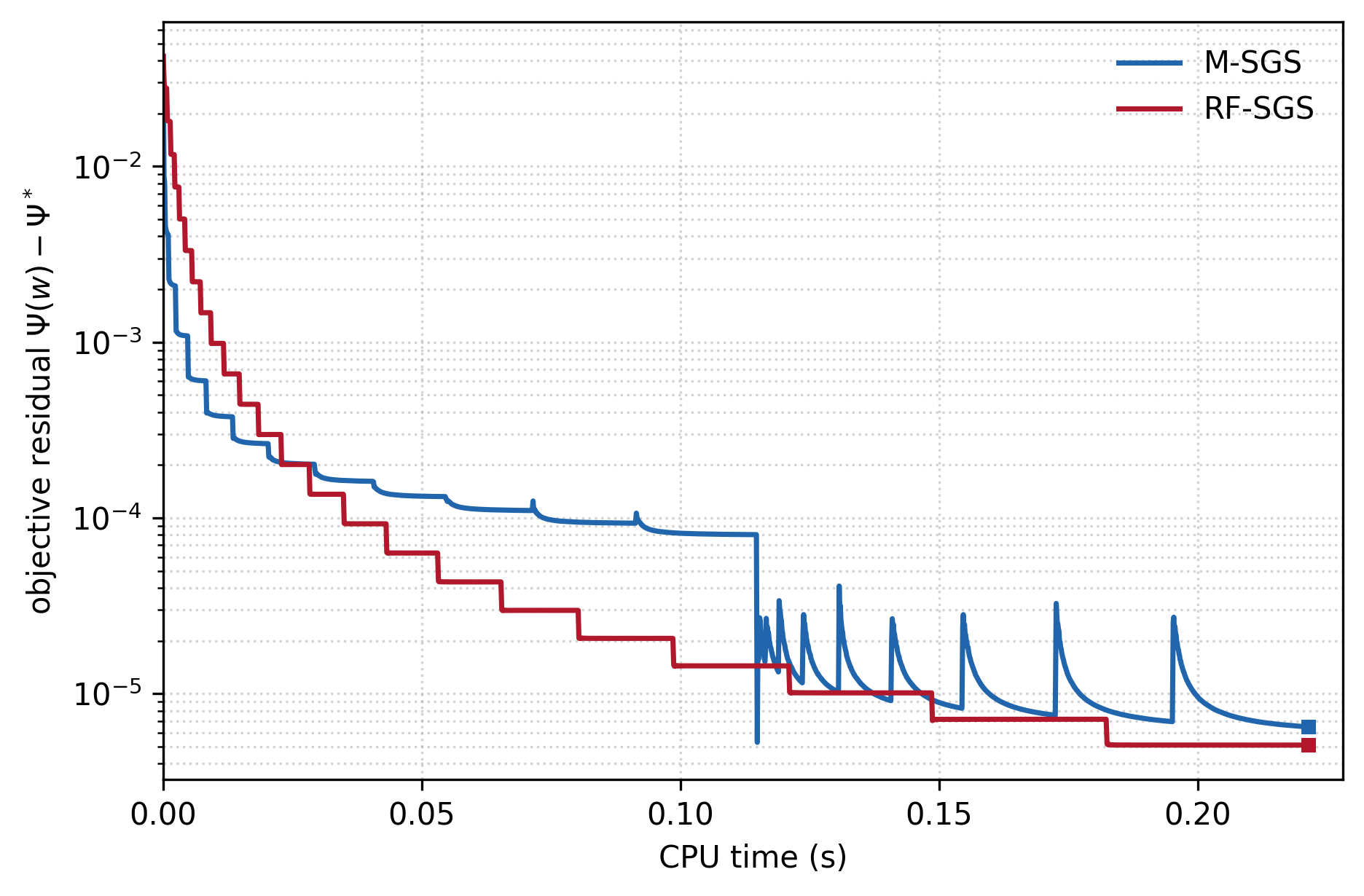}}
			&\subfigure[CPU time; $n=2000$]{
				\label{fig:portfolio-cpu-b}
				\includegraphics[width=0.43\textwidth]{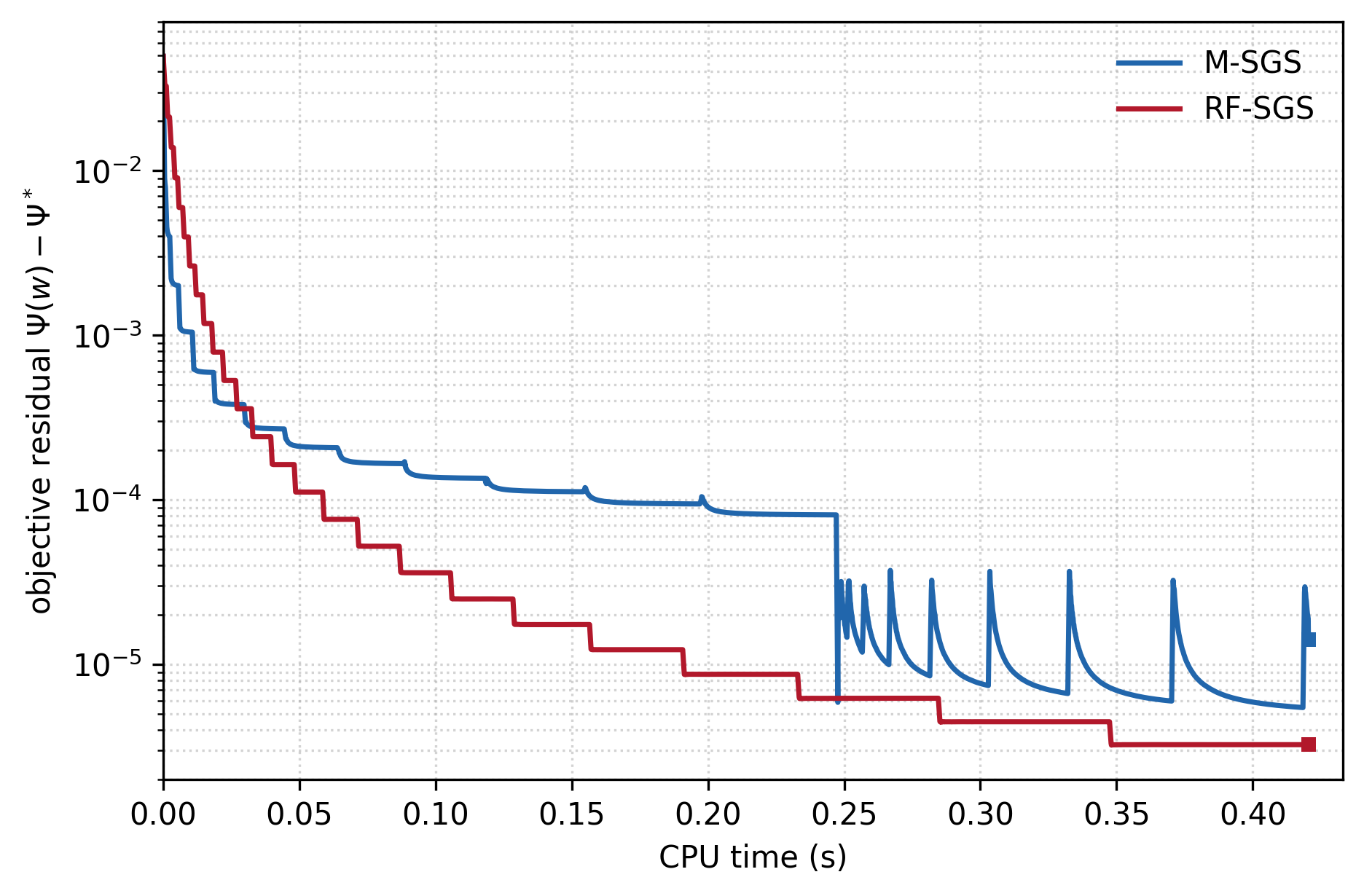}}
		\end{tabular}
		\caption{Comparison of RF-SGS and M-SGS on the portfolio optimization problem.\label{Fig.4.1}}
	\end{figure}
	
	Figure~\ref{Fig.4.1} reports the objective residual $\Psi(w)-\Psi^*$ at the same feasible candidate after each inner update. The upper panels use cumulative inner updates, and the lower panels use the corresponding CPU time. RF-SGS produces a smoother decrease than M-SGS, whose trajectory reflects the periodic restarts. In both dimensions and under both performance measures, RF-SGS reaches a lower residual within the common budget. This demonstrates that the restart-free parameter schedule preserves the practical efficiency of M-SGS while avoiding the oscillations and implementation overhead caused by explicit restarts.
	
	\subsection{Generalized Fused-Lasso Regression}
	
	We next consider the generalized fused-lasso model \cite{Tibshirani2005}
	\begin{align}\label{exp:fused-model}
		\min_{x\in\mathbb R^p}\;
		\Psi(x)
		:=\frac{1}{2m}\|Ax-b\|_2^2
		+\max_{y\in Y}\langle Dx,y\rangle
		+\frac{\mu}{2}\|x\|_2^2,
	\end{align}
	where $A\in\mathbb R^{m\times p}$ is the design matrix, $b\in\mathbb R^m$ is the response vector, and $D$ is the incidence matrix of a path graph. Its rows are $d_e=e_{e+1}-e_e$, $e_i$ is the $i$th standard basis vector, and $Y=\left[-\frac{\lambda_{\rm G}}q,\frac{\lambda_{\rm G}}q\right]^q$, $q=p-1.$
	We use Euclidean distance-generating functions:
	\begin{align*}
		\omega(x)&:=\frac12\|x\|_2^2,\qquad
		V(a,b):=\frac12\|b-a\|_2^2,\qquad
		d(y):=\frac12\|y\|_2^2,\\
		\mu_s&=\nu=1,\qquad
		B_h=\Omega=\frac{\lambda_{\rm G}^2}{2q},\qquad
		C_h=\|D\|^2.
	\end{align*}
	
	We compare RF-ASSGS with M-AGS \cite{Lan2022} under exact and stochastic access to $\nabla h_\eta$. We set $p=1000$, $m=3000$, $\lambda_{\rm G}=0.4$, $\mu=0.08$, $\epsilon=10^{-3}$. 
	The entries of $A$ are sampled independently from $\mathcal N(0,1)$, and
	$b=Ax^{\rm true}+0.08\zeta$, where $x^{\rm true}$ is piecewise constant with $\|x^{\rm true}\|_2=1$ and $\zeta\sim\mathcal N(0,I_m)$ is fixed. Following \cite{Lan2022}, we estimate $\Psi^*$ using an independent deterministic AGS reference run with $\eta_{\rm ref}=10^{-4}$ and $N_f^{\rm ref}=2000$. RF-ASSGS uses the parameters from Theorem~\ref{thm3.5}, while M-AGS uses those in \cite{Lan2022}.
	
	\begin{figure}[!htp]
		\centering
		\subfigure[Inner iterations.]{
			\includegraphics[width=0.48\textwidth]{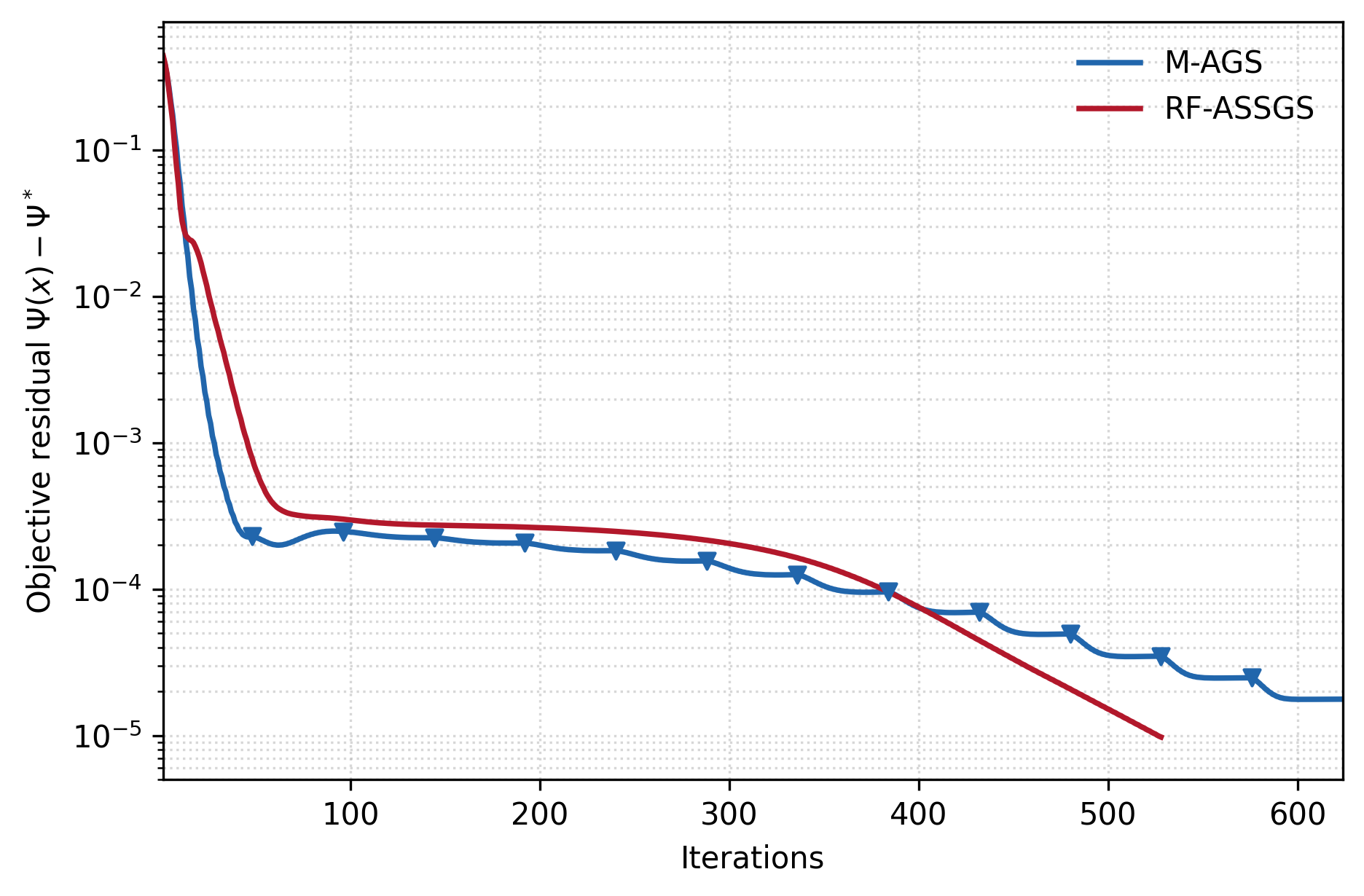}}
		\subfigure[CPU time.]{
			\includegraphics[width=0.48\textwidth]{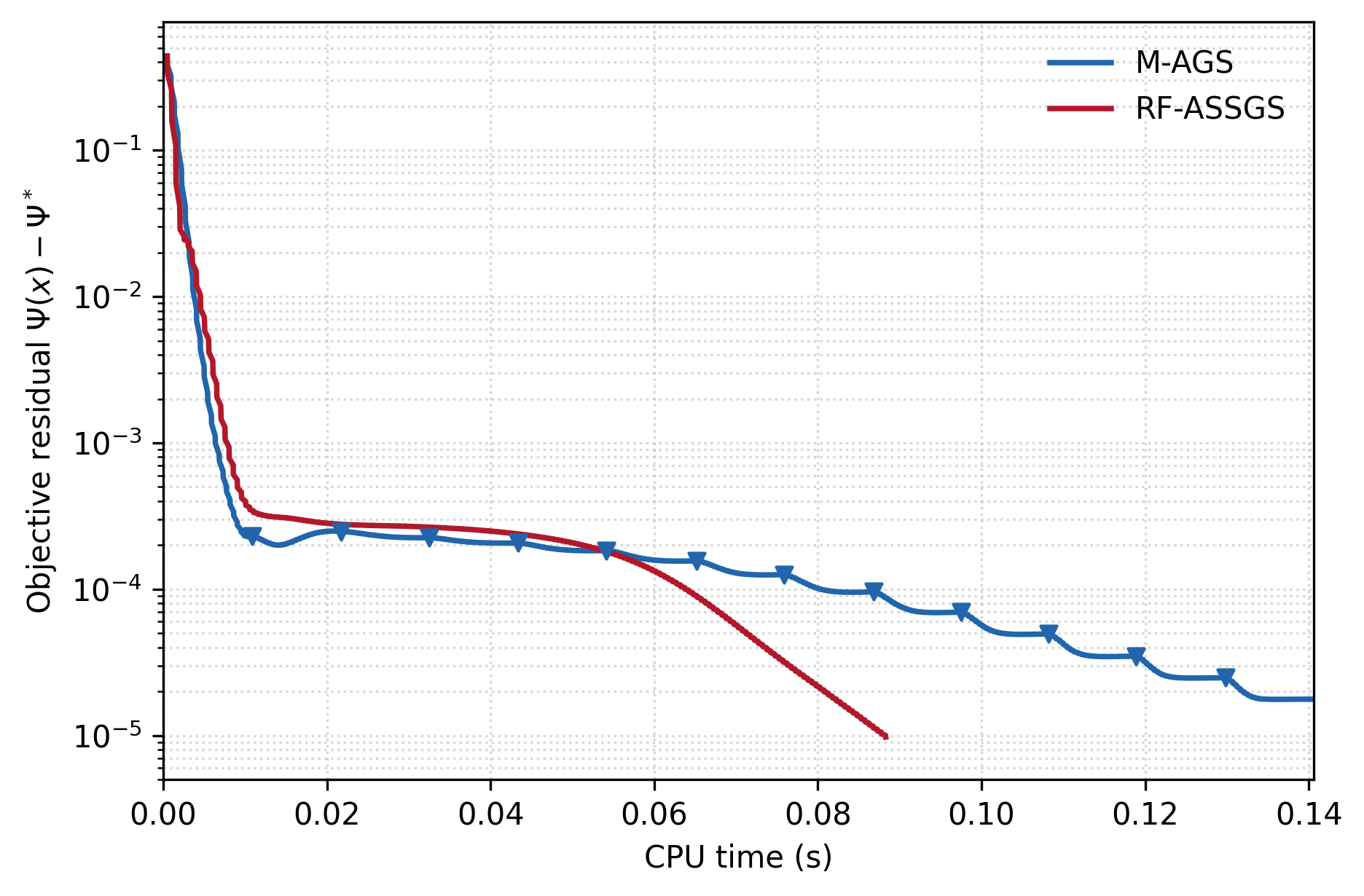}}
		\caption{Deterministic comparison of RF-ASSGS and M-AGS.\label{fig:fused-exact-candidate}}
	\end{figure}
	
	\begin{table}[!htp]
		\centering
		\begin{tabular}{lcccc}
			\hline
			Algorithm & Terminal residual & $N_K$ & $N_f$ & CPU time (s)\\
			\hline
			RF-ASSGS & $9.77\times10^{-6}$ & $528$ & $176$ & $0.0883$\\
			M-AGS    & $1.77\times10^{-5}$ & $624$ & $312$ & $0.1406$\\
			\hline
		\end{tabular}
		\caption{Exact-oracle results for the complete theorem-scheduled runs.}
		\label{tab:fused-exact-results}
	\end{table}
	
	As shown in Table~\ref{tab:fused-exact-results}, RF-ASSGS achieves a lower terminal residual while using fewer evaluations of $\nabla f$ and less CPU time. In this experiment,
	$\nabla f(x)=\frac1mA^{\!\top}(Ax-b),$
	whose evaluation costs $O(mp)$ for dense $A$. In contrast, one application of $D$ or $D^\top$ costs only $O(p)$. Thus, reducing $N_f$ directly improves the computational efficiency, which explains the advantage of RF-ASSGS in both oracle usage and CPU time.
	
	We finally compare the methods with the stochastic oracle $H_\eta$, using one sample per inner update. RF-ASSGS again uses the parameters from Theorem~\ref{thm3.5}. For M-AGS, let  $L_{\rm M}:=L+\mu$ and $M_s:=\max\{L_{\rm M},\|D\|^2/\eta_s\}$. 
	The curve ``M-AGS (original parameters)'' uses the schedule and parameters of \cite{Lan2022}, with the exact gradient replaced by $H_\eta$. The curve ``M-AGS($\widehat M_s=256M_s$)'' uses the same schedule with the conservative estimate $\widehat M_s=256M_s$.
	
	\begin{figure}[!htp]
		\centering
		\subfigure[Inner iterations.]{
			\includegraphics[width=0.48\textwidth]{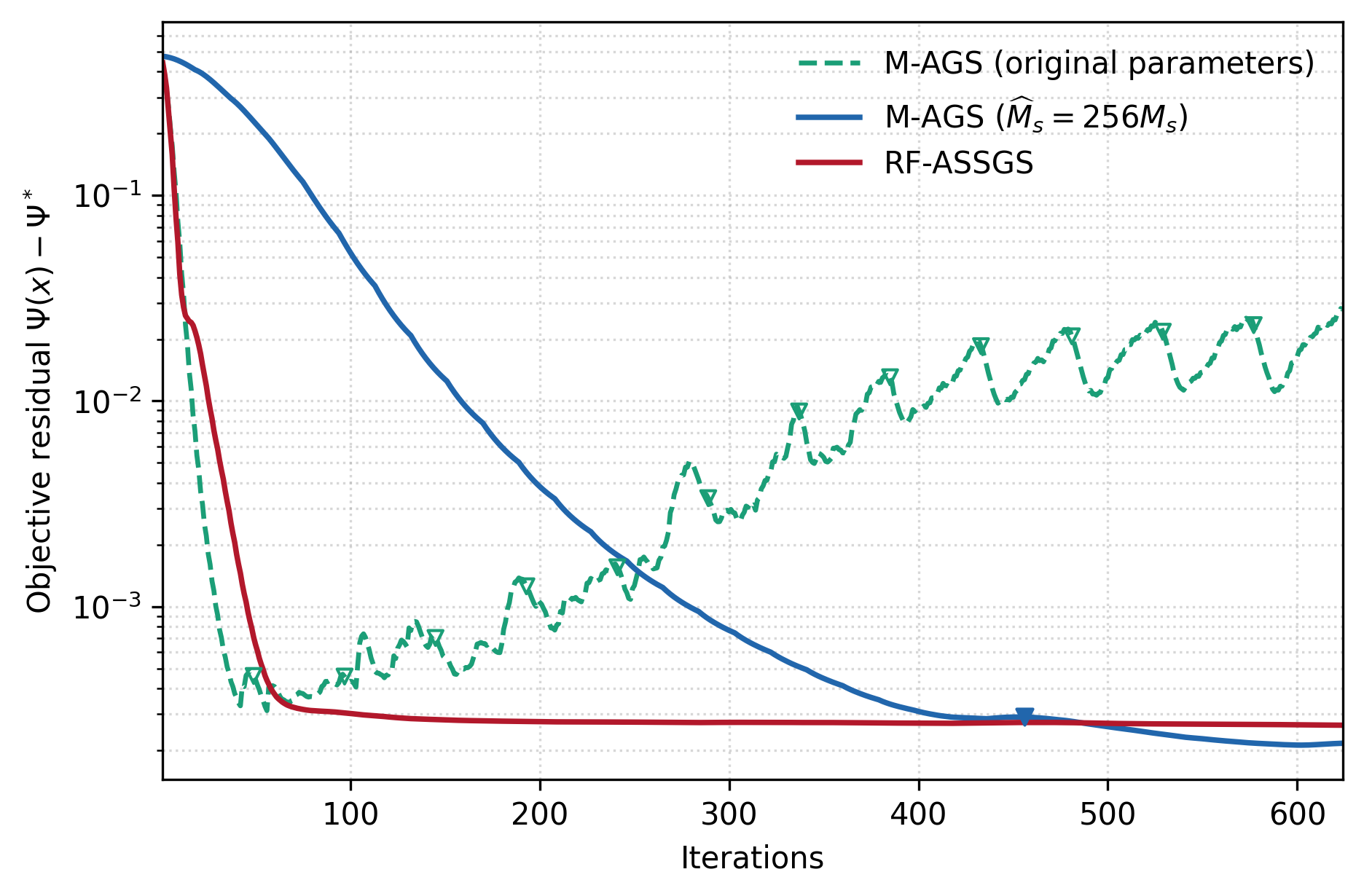}}
		\subfigure[CPU time.]{
			\includegraphics[width=0.48\textwidth]{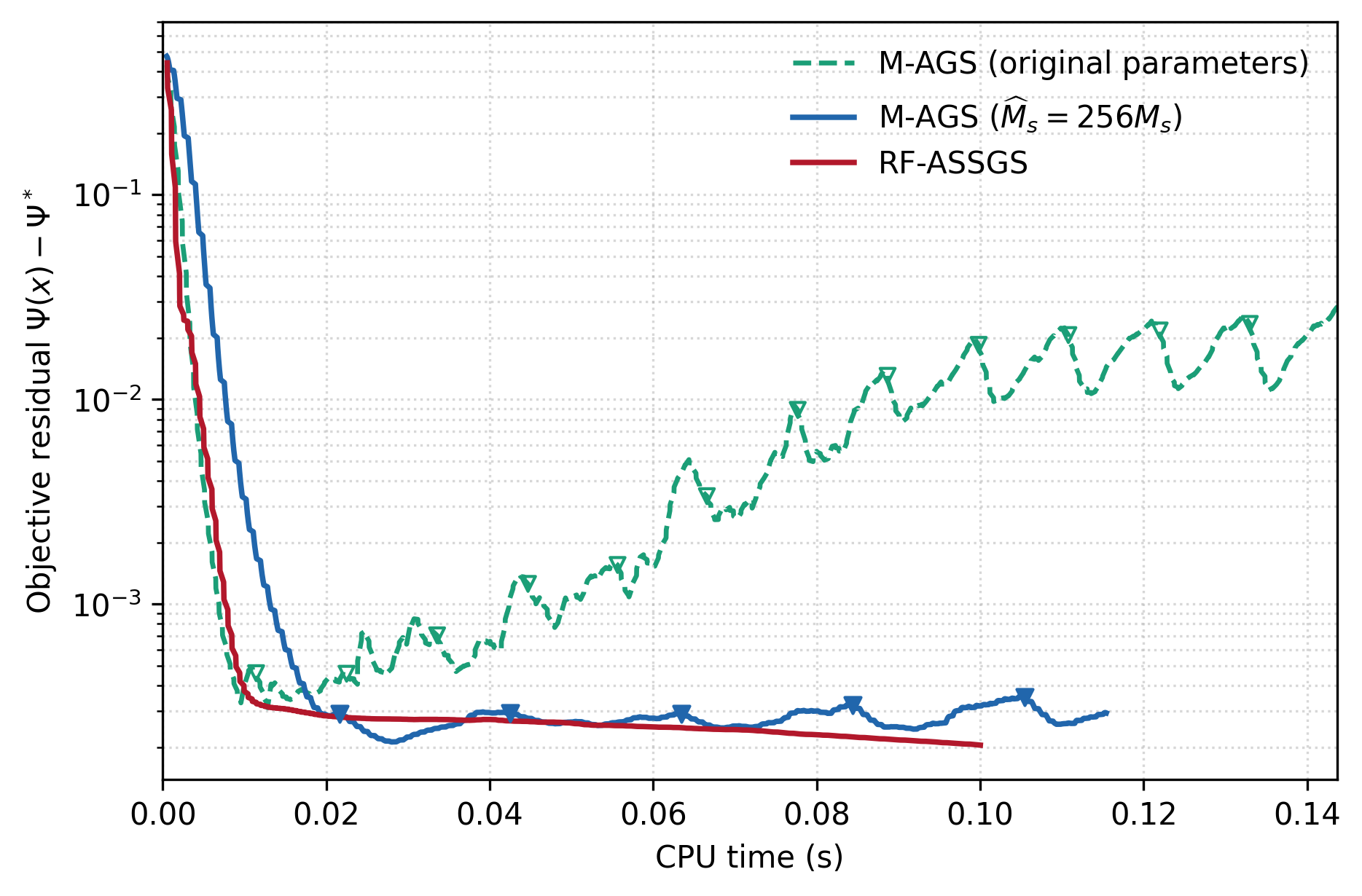}}
		\caption{Stochastic comparison of RF-ASSGS and M-AGS.\label{fig:fused-stochastic-candidate}}
	\end{figure}
	
	Figure~\ref{fig:fused-stochastic-candidate} shows that RF-ASSGS attains a lower residual than both M-AGS variants under the same inner-update budget, while requiring less CPU time. The original M-AGS parameters were designed for exact gradients and do not account for the variance of a single-sample oracle, resulting in an unstable trajectory after direct stochastic substitution. Although the conservative estimate $\widehat M_s$ improves the stability of M-AGS, it increases the inner-loop length and does not eliminate the advantage of RF-ASSGS. These results confirm that the restart-free parameter schedule is effective not only with exact gradients but also with stochastic first-order information.

\section{Conclusion}

We studied restart-free gradient sliding methods for strongly convex composite optimization problems whose objectives consist of a smooth component, a general nonsmooth component, and a relatively simple nonsmooth term. The main contributions address two complementary oracle settings.

First, when stochastic subgradients of the general nonsmooth component are available, we developed the restart-free stochastic gradient sliding (RF-SGS) method. By incorporating strong convexity through a novel parameter-selection strategy, RF-SGS computes an $\epsilon$-solution using $\mathcal{O}(\log(1/\epsilon))$ evaluations of the gradient of the smooth component and $\mathcal{O}(1/\epsilon)$ stochastic subgradient evaluations of the general nonsmooth component, up to the problem-dependent factors specified in the complexity bounds. These orders match those of existing multiphase, restart-based methods, while RF-SGS requires neither explicit restarts nor a multiphase implementation.

Second, when stochastic subgradients are unavailable or expensive to compute, but the general nonsmooth component admits a compatible smooth approximation and an associated first-order oracle, we developed the restart-free accelerated stochastic smoothed-gradient sliding (RF-ASSGS) method. Its complexity is $\mathcal{O}(\log(1/\epsilon))$ gradient evaluations of the smooth component and $\mathcal{O}(1/\sqrt{\epsilon}+\sigma^2/\epsilon)$ stochastic-oracle evaluations, again up to problem-dependent factors. This result covers, in particular, max-form nonsmooth functions and remains applicable whether strong convexity is carried by the relatively simple term or, through the equivalent reformulation in Remark~\ref{remark:2}, by the smooth component.

Numerical experiments on portfolio optimization and generalized fused-lasso regression support these theoretical results. In both the subgradient and smoothing-based settings, the proposed restart-free methods achieve convergence performance comparable to or better than their restart-based counterparts. Overall, the results show that restart-free gradient sliding methods can retain the relevant strong-convexity complexity guarantees in both oracle settings while substantially simplifying the algorithmic structure and implementation.

	%
\end{document}